\documentclass[11pt,a4paper,leqno]{article}
\usepackage[textwidth=16cm, textheight=23cm,marginratio=1:1]{geometry}
\usepackage{amsmath}
\usepackage{amsfonts}
\usepackage{amssymb}
\usepackage{tikz-cd}
\usepackage{amsthm}
\usepackage{booktabs} 
\usepackage[textwidth=0.9in]{todonotes}

\usepackage{caption} 
\usepackage{enumitem}
\usepackage[breaklinks,bookmarks=false, hidelinks]{hyperref} 
\hypersetup{
	colorlinks,
	linkcolor={red}
}

\usepackage{libertine}
\usepackage{libertinust1math} 

\makeatletter
\let\c@table\c@figure 
\let\ftype@table\ftype@figure 
\makeatother

\numberwithin{equation}{section}

\newtheorem{theorem}[equation]{Theorem}
\newtheorem{lemma}[equation]{Lemma}
\newtheorem{proposition}[equation]{Proposition}
\newtheorem{corollary}[equation]{Corollary}
\newtheorem*{theorem*}{Theorem}

\theoremstyle{definition}
\newtheorem{definition}[equation]{Definition}
\newtheorem{remark}[equation]{Remark}

\newtheorem*{question*}{Question}

\theoremstyle{remark}
\newtheorem{example}[equation]{Example}

\DeclareMathOperator{\im}{Im}

\DeclareMathOperator{\Aut}{Aut}

\DeclareMathOperator{\Hom}{Hom}

\DeclareMathOperator{\Sym}{Sym}

\DeclareMathOperator{\rad}{rad}

\DeclareMathOperator{\Soc}{Soc}

\DeclareMathOperator{\Gr}{Gr}

\DeclareMathOperator{\GL}{GL}

\DeclareMathOperator{\gr}{gr}
\DeclareMathOperator{\Spec}{Spec}
\DeclareMathOperator{\Ann}{Ann}
\DeclareMathOperator{\Hilb}{Hilb}

\DeclareMathOperator{\id}{id}

\DeclareMathOperator{\pr}{pr}

\DeclareMathOperator{\inn}{in}

\newcommand{\OO}{\mathcal{O}}%
\newcommand{\PP}{\mathbb{P}}
\newcommand{\KK}{\mathbb{k}}
\newcommand{\inj}{\hookrightarrow}
\renewcommand{\AA}{\mathbb{A}}
\newcommand{\CC}{\mathbb{C}}
\newcommand{\BBname}{Bia{\l{}}ynicki-Birula}
\newcommand{\Gmult}{\mathbb{G}_{\mathrm{m}}}%
\newcommand{\HilbFunc}[2]{\Hilb_{#1}(\AA^{#2},0)}
\newcommand{\HilbFuncTheta}[2]{\widetilde{\Hilb}_{#1}(\AA^{#2},0)}
\newcommand{\HilbFuncGor}[2]{\Hilb_{#1}^{\mathrm{Gor}}(\AA^{#2},0)}

\newcommand{\HilbFuncGr}[2]{\Hilb^{ \Gmult}_{#1}(\AA^{#2},0)}
\newcommand{\HilbSmk}{\Hilb^{\mathrm{sm}}_{k}}
\newcommand{\EE}{\mathcal{E}}%
\newcommand{\into}{\hookrightarrow}
\newcommand{\onto}{\twoheadrightarrow}
\newcommand{\spann}[1]{\left\langle #1 \right\rangle}
\newcommand{\goodquotient}{\mathbin{
  \mathchoice{\left/\mkern-8mu\right/}
    {/\mkern-7mu/}
    {/\mkern-7mu/}
    {/\mkern-7mu/}}}
\newcommand{\mm}{\mathfrak{m}}%

\begin{document}

\title{On construction of $k$-regular maps to Grassmannians via algebras of socle dimension two}
\author{Joachim Jelisiejew\thanks{Faculty of Mathematics, Informatics
and Mechanics, 
Stefana Banacha 2, 02-097 Warszawa, Poland. Supported by Polish
NCN grant 2020/39/D/ST1/00132. A revision of this work has partially supported
by the Thematic Research Programme "Tensors: geometry, complexity and quantum entanglement", University of Warsaw, Excellence Initiative – Research University and the Simons Foundation Award No. 663281 granted to the Institute of Mathematics of the Polish Academy of Sciences for the years 2021-2023.},
Hanieh Keneshlou\thanks{Department of Mathematics and Statistics, Universitätsstraße 10, 78464 Konstanz, Germany.}}

\maketitle

\begin{abstract}
    A continuous map $\mathbb{C}^n\to \Gr(\tau, N)$ is $k$-regular if the
    $\tau$-dimensional subspaces corresponding to images of any $k$ distinct points
    span a $\tau k$-dimensional space. For $\tau = 1$ this essentially
    recovers the classical notion of a $k$-regular map $\mathbb{C}^n\to
    \mathbb{C}^N$. We provide new examples of $k$-regular
    maps, both in the classical setting $\tau = 1$ and for $\tau\geq 2$, where
    these are the first examples known.\\
    Our methods come from algebraic geometry, following and
    generalizing~\cite{bjjm}. The key and highly nontrivial part of the
    argument is proving that certain loci of the Hilbert scheme of points have
    expected dimension. As an important side result, we prove
    the irreducibility of the punctual Hilbert scheme of $k$ points on a
    threefold, for $k\leq 11$.
\end{abstract}

\tableofcontents


\section{Introduction}

    A continuous map $f\colon \mathbb{C}^n\to \mathbb{C}^N$ is (linearly)
    \emph{$k$-regular} if
    the images $f(x_1)$, \ldots, $f(x_k)$ are linearly independent for any
    pairwise distinct points $x_1, \ldots ,x_k$. The problem of the existence of
    such maps for given $k$, $n$, $N$ is researched classically in topology
    and is also important in applications, for example to approximation theory,
    see~\cite{bjjm}.
    This problem decomposes into two distinctly differently flavored halves:
    providing \emph{lower} and \emph{upper} bounds for the minimal $N = N(n,
    k)$ for which a $k$-regular map exists.
    The lower bounds are provided using algebro-topological
    methods~\cite{Ziegler, Ziegler2}, while the upper bounds follow from
    constructions in~\cite{bjjm}.

    A continuous map $f\colon \mathbb{C}^n\to \mathbb{P}(\CC^N)$ is
    \emph{projectively $k$-regular} if the images $f(x_1)$, \ldots, $f(x_k)$
    span a $(k-1)$-dimensional projective subspace for any
    pairwise distinct points $x_1, \ldots ,x_k$. This notion is closely related
    to linear $k$-regularity, see~\cite[Lemma~2.3]{bjjm}.

    A natural generalization of projectively $k$-regular maps comes from
    replacing $\mathbb{P}(\mathbb{C}^{N})$ by a Grassmannian $\Gr(\tau, N)$
    of $\tau$-dimensional subspaces of $\mathbb{C}^N$. A continuous map
    $f\colon
    \mathbb{C}^n\to \Gr(\tau, N)$ is \emph{$k$-regular} if for every pairwise distinct
    points $x_1, \ldots ,x_k\in \CC^n$ the subspaces $f(x_1), \ldots ,f(x_k)$
    jointly span a $(k\cdot \tau)$-dimensional subspace of $\mathbb{C}^N$,
    i.e., their basis vectors are jointly linearly independent. For $\tau = 1$
    this recovers the notion of a projectively $k$-regular map. Also in this
    context one would like to bound the minimal $N=N(\tau, n, k)$ for which a
    $k$-regular map exists.\\

    In this article, we prove the following result.
    \newcommand{\boundingFunction}[3]{\widetilde{N}(#1, #2, #3)}%
     \begin{theorem}\label{ref:mainthm}
        Consider the function $\boundingFunction{\tau}{k}{n}$ defined by
        \begin{enumerate}
            \item\label{it:mainthm:one} If $k\leq 8$, or if $\tau\leq2$ and $k\leq 11$, then
                $\boundingFunction{\tau}{k}{n} := (n-1)(k-1) + \tau k$.
            \item Otherwise $\boundingFunction{\tau}{k}{n} = n(k-1)-1 + \tau k$.
        \end{enumerate}
        Then there exists a $k$-regular map $\mathbb{C}^n\to \Gr(\tau,
        \boundingFunction{\tau}{k}{n})$.
    \end{theorem}
    In the classical case $\tau = 1$ this extends the bounds from~\cite{bjjm}
    for $k = 10,11$.
    The method of construction, which follows the ideas of~\cite{bjjm}, is the following.
    \begin{enumerate}
        \item By a direct construction, a $k$-regular algebraic map $f_0\colon \mathbb{C}^n\to
            \Gr(\tau, N_0)$ exists for $N_0\gg 0$.
        \item We seek a $W\subset \mathbb{C}^{N_0}$ such that the composed map
            $f_W\colon \mathbb{C}^n\to \Gr(\tau, N_0)\dashrightarrow \Gr(\tau, N_0 - \dim
            W)$ induced by $\mathbb{C}^{N_0}\onto \mathbb{C}^{N_0}/W$ is still
            everywhere defined and $k$-regular. In fact, since
            $\mathbb{C}^n$ is homeomorphic to any open ball around the origin,
            it is enough for $f_W$ to be defined and $k$-regular around the
            origin.
        \item A subspace $W$ satisfies the requirements above if it does not
            intersect a certain \emph{bad locus}, given by the spans of all
            $f_0(Z)$ for $Z\subset \mathbb{C}^n$ a zero-dimensional degree
            $k$ subscheme supported at the
            origin. The dimension of this bad locus is bounded from above by
            the dimension of the family of possible $Z$, which is the punctual Hilbert scheme
            $\Hilb_k(\mathbb{A}^n, 0)$, see below.
        \item The bad locus above is covered by spans of $Z$ which have socle
            dimension $\leq \tau$, so we may restrict to these. This is
            perhaps geometrically obscure, but
            allows reducing to
            bounding the dimension of the locus $\Hilb_k^{\tau}(\mathbb{A}^n,
            0) \subset \Hilb_k(\mathbb{A}^n, 0)$ parameterizing them.
    \end{enumerate}
    The final remaining task, to bound $\dim \Hilb_k^{\tau}(\mathbb{A}^n, 0)$,
    is completely disconnected from topology, but very far from trivial,
    especially to get the better estimate in Point~\ref{it:mainthm:one}.
    It occupies the major part of the current paper and we discuss it at length below.
    Let us only remark that the above Theorem~\ref{ref:mainthm} is close to
    optimal when employing the methods of~\cite{bjjm}, as follows from
    Theorem~\ref{ref:dimension:maintheorem} below. Hence radically new
    approaches are necessary to improve the upper bounds on $N(\tau, n,
    k)$ further.
    We also remark that the case $\tau = 1$ is qualitatively easier than the
    other ones, thanks to the fact that in this case one deals exclusively
    with Gorenstein algebras, see~\S\ref{sec:regularMaps} for details,
    and a mature theory of classification of those algebras exists.

    We work over an algebraically closed field
    $\KK$ of characteristic zero.
    The dimension of the bad locus above is governed by
    a sublocus of the Hilbert scheme of points
    $\Hilb_k(\mathbb{A}^n)$.
    The Hilbert scheme is a quasi-projective scheme and its closed points
    are in bijection with closed subschemes $Z \subset
    \mathbb{A}^n$ that are zero-dimensional and of degree $k$; the simplest
    example of $Z$ is
    a tuple of $k$ closed points of $\mathbb{A}^n$. The topology of the
    Hilbert scheme is
    defined functorially, see for example \cite[\S4]{bjjm} for details. We
    denote by $[Z]\in \Hilb_k(\mathbb{A}^n)$ the point corresponding to
    $Z\subseteq \mathbb{A}^n$.

    Let $\Hilb_k(\mathbb{A}^n, 0)\inj \Hilb_k(\mathbb{A}^n)$ denote the
    closed locus consisting of $Z \subset \mathbb{A}^n$ such that $Z$ is
    supported only at the origin. This is the \emph{punctual} Hilbert scheme,
    see~\S\ref{sec:irreducibilitypreliminaries} for details. It is a
    projective, in particular proper, scheme. (A warning: some authors use
    the name ``punctual'' for the whole $\Hilb_k(\mathbb{A}^n)$.)

    Compared with $\Hilb_k(\mathbb{A}^n)$, much less is known
    about the punctual Hilbert scheme $\Hilb_k(\mathbb{A}^n, 0)$.
    This scheme has a distinguished irreducible component, the
    \emph{curvilinear locus}
    which is the locus of \emph{curvilinear schemes}: schemes
    isomorphic to $\Spec(\KK[\varepsilon]/\varepsilon^k)$,
    see~\cite{ia_deformations_of_CI}. This locus has
    dimension $(n-1)(k-1)$, thus $\dim \Hilb_k(\mathbb{A}^n, 0) \geq
    (n-1)(k-1)$. We say that $(n-1)(k-1)$ is the \emph{expected dimension} of
    $\Hilb_k(\mathbb{A}^n,0)$ and  that a locus of $\Hilb_k(\mathbb{A}^n,0)$ is \textit{negligible} if
    its dimension is at most $(n-1)(k-1)$.
    Let $\Hilb_k^{\tau}(\mathbb{A}^n,
    0)\subset \Hilb_k(\mathbb{A}^n,0)$ denote the open locus that consists of
    $[Z = \Spec(A)]$ where the socle of $A$ is at most $\tau$-dimensional;
    for example $\tau=1$ corresponds to a Gorenstein $Z$.
    Our first main result concerns $(n,k,\tau)$ for which the dimension of
    $\Hilb_k^{\tau}(\mathbb{A}^n, 0)$ is
    expected.

    \begin{theorem}\label{ref:dimension:maintheorem}
        \def\checkmark{\tikz\fill[scale=0.4](0,.35) -- (.25,0) -- (1,.7) --
        (.25,.15) -- cycle;}
        For every $k\leq 8$ and $n$ the scheme $\Hilb_k(\mathbb{A}^n,0)$ has expected dimension
        while for every $k\geq 9$ the dimension of $\Hilb_k^3(\mathbb{A}^{k-4}, 0)$ is higher than
        expected.
        For every $k\leq 11$ and $n$ the scheme $\Hilb_k^2(\mathbb{A}^n,0)$ has
        expected dimension, while for $k\geq 13$ the dimension of
        $\Hilb_{k}^2(\mathbb{A}^{\lfloor k/2\rfloor-1},0)$ is higher than
        expected and
        for $k\geq 14$ the dimension of
        $\Hilb_{k}^1(\mathbb{A}^{\lfloor k/2\rfloor-1},0)$ is higher than
        expected. This information can be summarized as follows, where
        \emph{no} means \emph{no if $n$ is big enough}.

        \begin{center}
            \begin{tabular}[<+position+>]{l c c c c c}
                & $k\leq 8$ & $9\leq k\leq 11$ & $k=12$ & $k=13$ & $k\geq 14$\\
                \toprule
                $\Hilb_k(\mathbb{A}^n, 0)$ & \checkmark & no & no & no & no \\
                $\Hilb^{\tau}_k(\mathbb{A}^n, 0)$, $\tau\geq 3$ & \checkmark & no & no & no & no \\
                $\Hilb^2_k(\mathbb{A}^n, 0)$ & \checkmark & \checkmark & ? & no & no \\
                $\Hilb^1_k(\mathbb{A}^n, 0)$ & \checkmark & \checkmark & ? & ? & no
            \end{tabular}\\\vspace{1mm}
            {Table. \small Does $\Hilb^{\tau}_k(\mathbb{A}^n, 0)$ have expected
        dimension for every $n$?}
    \end{center}
    \end{theorem}
    Before the current work, the case $\tau = 1$ was investigated in~\cite[Theorem~A.16]{bjjm}
    where it was proven that $\dim\Hilb^1_k(\mathbb{A}^n, 0) = (k-1)(n-1)$ for $k\leq 9$.
    Apart from this, to the authors' knowledge, no results were known for
    general $n$.
    For $n\leq 2$ the closure of the curvilinear locus is the only component; the punctual
    Hilbert scheme is irreducible by a result of Brian\c{c}on~\cite{briancon},
    see also~\cite{iarrobino_punctual}.
    In~\cite[Corollary~(1.2)]{Ellingsrud_Stromme__On_the_homology} the authors
    reprove Brian\c{c}on's result via a beautiful topological argument, which relies heavily
    on that $\Hilb_k(\mathbb{A}^2)$ is smooth and irreducible.
    In the current article, we extend the irreducibility result to threefolds, for $k\leq 11$.
    \begin{theorem}\label{ref:irreducibility}
        For $k\leq 11$ the scheme $\Hilb_k(\mathbb{A}^3, 0)$ is equal to the
        closure of the curvilinear locus, so
        it is irreducible of dimension $2(k-1)$.
    \end{theorem}
    The punctual Hilbert scheme does not depend on the
    global geometry~\cite[Proposition~2.2]{fogarty} so we immediately
    obtain the following generalization.
    \begin{theorem}\label{ref:irreducibilityGeneral}
        Let $x\in X$ be a smooth point on a threefold  $X$.
        Then for $k\leq 11$ the scheme $\Hilb_k(X, x)$ is irreducible of dimension $2(k-1)$.
    \end{theorem}

    It was known that $\Hilb_k(\mathbb{A}^3)$ is irreducible for
    $k\leq 11$~\cite{DJNT, Henni, Sivic__Varieties_of_commuting_matrices}.
    However, this by no means implies irreducibility of $\Hilb_k(\mathbb{A}^3, 0)$.
    Indeed, the proof of irreducibility of $\Hilb_k(\mathbb{A}^3)$ as in~\cite{DJNT} proceeds by proving
    that every scheme $Z\subset \mathbb{A}^3$ is a degeneration of
    $\Gamma\subset \mathbb{A}^3$ which is a union of $k$ points. Obviously,
    the scheme $\Hilb_k(\mathbb{A}^3, 0)$ contains no such $\Gamma$, so the
    picture is as in Figure~\ref{fig:Hilbdiffs}. None of the known
    general techniques used to prove irreducibility of Hilbert schemes adapts
    effectively to analysing irreducibility of $\Hilb_k(\mathbb{A}^3, 0)$ and
    here we develop a new method.

    \begin{figure}[h]
        \centering
        \includegraphics[width=6cm]{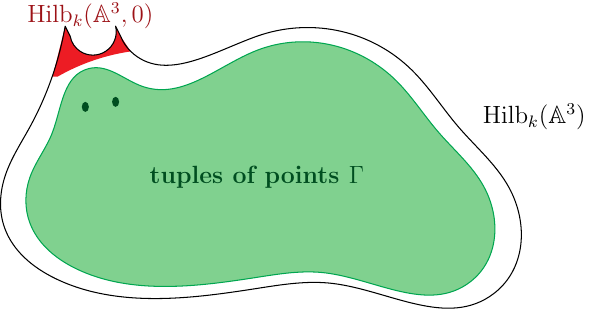}
        \caption{The schemes $\Hilb_k(\mathbb{A}^3)$ and
        $\Hilb_k(\mathbb{A}^3, 0)$, $k\leq 11$.}
        \label{fig:Hilbdiffs}
    \end{figure}

    \subsection{Proof ideas I: general and \BBname{}
slicing}\label{ssec:proofIdeas}
    The major part of the proof of Theorem~\ref{ref:irreducibility} is to
    prove an appropriate part of Theorem~\ref{ref:dimension:maintheorem}: to show
    that the dimension of $\Hilb_k(\mathbb{A}^3, 0)$ is at most $2(k-1)$ and
    additionally that there is at most one irreducible component of dimension
    $2(k-1)$. Since the dimension of $\Hilb_k(\mathbb{A}^3, 0)$
    at every point is \emph{at least} $2(k-1)$ by general nonsense, see
    Proposition~\ref{ref:componentsLowerBound:prop}, this implies that this
    component is the whole
    $\Hilb_k(\mathbb{A}^3, 0)$.

    However, bounding the dimension of $\Hilb_k(\mathbb{A}^3, 0)$ is hard.
    The points of $\Hilb_k(\mathbb{A}^3, 0)$ are
    quite impossible to classify for $k > 6$ and the number of cases is best described as
    swarming, see tables below. Moreover, in contrast with $\Hilb_k(\mathbb{A}^n)$, the scheme
    $\Hilb_k(\mathbb{A}^n, 0)$ has not been much investigated.
    Therefore, the only one readily available estimate is ``dimension at a
    point is at most the dimension of the tangent space''.
    But even this method has a serious caveat. Namely, it is
    known classically that for a zero-dimensional $Z \subset \mathbb{A}^n$ we
    have
    \[
        T_{[Z]}\Hilb_k(\mathbb{A}^n) = \Hom_{\KK[x_1, \ldots , x_n]}(I, \KK[x_1, \ldots ,x_n]/I),
    \]
    where $I = I(Z) \subset  H^0(\OO_{\mathbb{A}^n}) = \KK[x_1,
    \ldots,x_n]$.
    However, for $[Z]\in \Hilb_k(\mathbb{A}^n, 0)$ no such nice
    description of the subspace $T_{[Z]} \Hilb(\mathbb{A}^n, 0)$ seems to exist,
    see~\cite{Bejleri_Stapleton} for the surface case. Given
    that $\Hilb_k(\mathbb{A}^3)$ is irreducible for $k\leq 11$, we have
    $\dim_{\KK} \Hom(I, \KK[x_1, \ldots ,x_n]/I) \geq 3k$ which is not a
    useful estimate.
    Even worse, the locus $\Hilb_k(\mathbb{A}^3, 0)$ is the most singular part of
    $\Hilb_k(\mathbb{A}^3)$; indeed every zero-dimensional $Z\subset \mathbb{A}^3$ can
    be degenerated to a scheme supported only at the origin by taking the
    limit at zero of the torus action on $\mathbb{A}^3$ by scalar
    multiplication. This implies
    that the tangent spaces to the Hilbert scheme at points of
    $\Hilb_k(\mathbb{A}^3, 0)$ are frequently of dimension much higher than
    $3k$.

    To provide an upper bound on the dimension, we introduce a technique that
    we baptize informally as
    \emph{\BBname{} slicing}. It seems applicable much
    more generally, outside the realm of Hilbert schemes, whenever one has a
    singular moduli space with a torus action. We discuss it here for
    arbitrary $\mathbb{A}^n$, not just $n=3$.

    Fix a one-dimensional torus $\Gmult$ and its action on
    $\mathbb{A}^n$ by $t\cdot (x_1, \ldots ,x_n) = (tx_1, \ldots ,
    tx_n)$. Every point except the origin
    diverges with $t\to \infty$.
    We consider the \BBname{}-decomposition $\Hilb_k^{+}(\mathbb{A}^n)$ of the Hilbert
    scheme, which subdivides this scheme by grouping together points that have
    their limit at
    infinity in the same connected component of the fixed points~\cite{Drinfeld,
    jelisiejew_sienkiewicz__BB}. The precise definition of the
    \BBname{} decomposition will not be important for us, below we gather all
    its relevant properties.
    The \BBname{} decomposition comes with morphisms of schemes
    \begin{equation}\label{eq:maindiagram}
        \begin{tikzcd}
            \Hilb_k^{+}(\mathbb{A}^n)\ar[r, "\theta"]\ar[d, "\pi"] &
            \Hilb_k(\mathbb{A}^n)\\
            \Hilb_k^{\Gmult}(\mathbb{A}^n)
        \end{tikzcd}
    \end{equation}
    where $\theta$ is the forgetful map, while $\pi$ maps a subscheme to its
    limit at $t = \infty$. Since schemes supported
    outside the origin do not admit a limit at infinity,
     the image of $\theta$
    is, as a set, $\Hilb_k(\mathbb{A}^n, 0)$~\cite[Proposition~3.3]{Jelisiejew__Elementary}.
    The closed subscheme $\Hilb_k(\mathbb{A}^n,
    0)$ is $\Gmult$-stable, hence inherits a decomposition
    $\Hilb_k^{+}(\mathbb{A}^n, 0)$ and  the corresponding maps. This subscheme is
    also projective, see~\S\ref{sec:irreducibilitypreliminaries}, so the forgetful map
    $\Hilb_k^{+}(\mathbb{A}^n, 0)\to \Hilb_k(\mathbb{A}^n, 0)$ is bijective on
    points and hence
    \[
        \dim \Hilb_k^{+}(\mathbb{A}^n, 0) = \dim \Hilb_k(\mathbb{A}^n, 0).
    \]
    The map $\pi$ restricted to $\Hilb_k(\mathbb{A}^n, 0)$ has an explicit
    description as the associated-graded-algebra map, see~\S\ref{ssec:ideas2}
    for details. In particular, all elements of a fiber of $\pi$ share the
    same Hilbert function.

    The
    points of $\Hilb_k^{\Gmult}(\mathbb{A}^n)$ correspond to subschemes $Z
    \subset \mathbb{A}^n$ whose ideals are homogeneous with respect to the
    standard grading. Thus, this scheme has at least as many
    connected components as there are Hilbert functions $H$:
    \begin{small}
        \begin{center}
        \begin{tabular}{c c c c c c c c c c c c c c}
            $k$ & $1$ & $2$ & $3$ & $4$ & $5$ & $6$ & $7$ & $8$ & $9$ & $10$ &
            $11$ 
            \\\toprule
            number of Hilbert functions $H$, with arbitrary $H(1)$: & $1$ & $1$ & $2$ & $3$ & $5$ &
            $8$ & $12$ & $18$ & $27$ & $40$ & $57$ 
        \end{tabular}
    \end{center}
    \end{small}
    Possible $H$ are classified by Macaulay~\cite[\S4]{BrunsHerzog}
    and their number grows exponentially in $n$.
    For any fixed Hilbert function $H$ with $\sum H = k$, consider the
    open-closed locus $\HilbFuncGr{H}{n} \subset
    \Hilb_k^{\Gmult}(\mathbb{A}^n)$ parameterizing points $[\Spec(A)]$ where
    $A$ has Hilbert function $H$, we call it the \emph{graded stratum}
    associated to $H$.
    Let $\HilbFunc{H}{n} \subset\Hilb_k^{+}(\mathbb{A}^n, 0)$ denote the locus
    parameterizing points $[\Spec(A)]$ where $A$ is a local algebra with
    Hilbert function $H$, this is the \emph{stratum} associated to the Hilbert
    function $H$. By the description of $\pi$ as the associated-graded-algebra
    map, we have $\HilbFunc{H}{n} = \pi^{-1}(\HilbFuncGr{H}{n})$.
    Moreover, $(\HilbFunc{H}{n})^{\Gmult} =
    \HilbFuncGr{H}{n}$ as the notation suggests.
    In Figure~\ref{fig:BBdecomposition} we give a very schematic description
    of its \BBname{} decomposition, compare Figure~\ref{fig:Hilbdiffs}.
    We would like to stress that the division of points of $\Hilb_k(\mathbb{A}^n, 0)$
    into strata using the Hilbert function is a very classical idea. The contribution of the
    \BBname{} decomposition is to give this classical idea a functorial
    interpretation. One important product of this interpretation is a formula
    for tangent space to the stratum, which we discuss now.

    \begin{figure}[h]
        \centering
        \includegraphics[width=5cm]{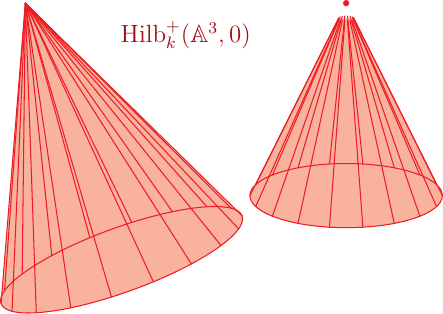}
        \caption{Schematic description of \BBname{} decomposition of the scheme $\Hilb^{+}_k(\mathbb{A}^n, 0)$.}
        \label{fig:BBdecomposition}
    \end{figure}

    For a $\Gmult$-fixed point
    $[Z]$ corresponding to a homogeneous ideal $I \subset R = \KK[x_1, x_2,  \ldots
    ,x_n]$ we have
    \[
        T_{[Z]}\Hilb_k^{+}(\mathbb{A}^n, 0) =
        (T_{[Z]}\Hilb_k(\mathbb{A}^n))_{\geq 0}  \simeq  \Hom_R(I, R/I)_{\geq
        0},
 \]
    where $\Hom(I, R/I)_{\geq 0}$ consists of homomorphisms $\varphi\colon I\to
    R/I$ such that $\varphi(I_i) \subset (R/I)_{\geq i}$ for every $i\geq 0$,~i.e., the homomorphisms that do not lower the degree, see for example~\cite[(2.2), (2.4)]{Jelisiejew__Elementary}.
    This restriction of the
    tangent space to its nonnegative part is the first key improvement. Indeed, when $n=3$, then $\dim_{\KK}\Hom(I,
    R/I)_{\geq0} \leq 2(k-1)$ for almost all homogeneous ideals,
    see~Proposition~\ref{ref:smallExceptionalStrata:prop} and the table below.
    To bound the dimension of the tangent spaces, we resort again to
    semicontinuity: choose a Borel subgroup $B \subset \GL_n$ that
    commutes with $\Gmult$. It acts
    on the components of the projective scheme
    $\Hilb^{\Gmult}_k(\mathbb{A}^n)$ and so by semicontinuity to bound the
    tangent space dimension
    (degree-wise) everywhere, it is enough to bound it for Borel-fixed points.
    Borel-fixed points correspond to monomial ideals, so there are finitely
    many of them for a fixed $k$. In fact, they correspond to very special monomial
    ideals~\cite{StronglyStableIdeals}. We list the numbers below, to provide
    some context.
    \begin{small}
    \begin{center}
    \begin{tabular}{c c c c c c c c c c c c c}
            \hfill $n=3$ \hfill $k = $ & $1$ & $2$ & $3$ & $4$ & $5$ & $6$ & $7$ & $8$ & $9$ & $10$ &
            $11$
            \\\toprule
            number of monomial ideals (MacMahon function) & $1$ & $3$ & $6$ &
            $13$ & $24$ & $48$ & $86$ & $160$ & $282$ & $500$ & $859$\\
           number of Borel-fixed ideals & $1$ & $1$ & $2$ & $3$ & $4$ & $6$ &
            $9$ & $12$ & $17$ & $24$ & $32$\\
            number of Borel-fixed ideals with $\dim T_{\geq 0}\geq 2(k-1)$ & $1$ & $1$ & $1$ &
            $1$ & $1$ & $1$ & $1$ & $2$ & $2$ & $4$ & $6$
        \end{tabular}
    \end{center}
\end{small}
    There remains a small number of special cases for $n = 3$ which have to be resolved by hand.
    There also remain many more problematic cases for $n > 3$, see below. To tackle those
    we need to employ other tools, such as the theory of
    Macaulay's inverse systems and partial classification results that will be discussed below.
    \begin{small}
    \begin{center}
    \begin{tabular}{c c c c c c c c c c c c c}
            \hfill $n = k$ \hfill $k = $ & $1$ & $2$ & $3$ & $4$ & $5$ & $6$ & $7$ & $8$ & $9$ & $10$ &
            $11$
            \\\toprule
            number of Borel-fixed ideals & $1$ & $1$ & $2$ & $3$ & $5$ & $8$ &
            $13$ & $20$ & $32$ & $50$ & $77$\\
            number of Borel-fixed ideals with $\dim T_{\geq 0}\geq (n-1)(k-1)$ & $1$ & $1$ & $1$ &
            $1$ & $1$ & $1$ & $1$ & $4$ & $8$ & $16$ & $33$
        \end{tabular}
    \end{center}
    \end{small}

    \emph{Dimension estimates.} For a $\Gmult$-fixed $[Z]$ given by a
    homogeneous $I\subset
    R = \KK[x_1, \ldots ,x_n]$ its tangent space
    \[
        T_{[Z]}:= \Hom_{R}(I, R/I)_{\geq 0}
    \]
    decomposes into the space $(T_{[Z]})_0$
    which is the tangent space to $\Hilb^{\Gmult}(\mathbb{A}^n)$
    and $(T_{[Z]})_{>0}$ which is the tangent space to the fiber over $[Z]$.
    Moreover, the fiber dimension of $\pi$ is upper-semicontinuous, because
    this map is a cone over a projective morphism, see~\cite[Proof of
    Theorem~3.1]{Szachniewicz} for details.
    Hence,
    we have the following estimates for the dimension of
    $\Hilb^+_k(\mathbb{A}^n, 0)$ near $[Z]$:
    \begin{enumerate}
        \item\emph{tangent space estimate} $\dim_{[Z]}
            \Hilb^+_k(\mathbb{A}^n, 0) \leq \dim_{\KK} \Hom_{R}(I,
            R/I)_{\geq 0}$,
        \item\emph{base-and-tangent-to-fiber estimate} $\dim_{[Z]}
            \Hilb^+_k(\mathbb{A}^n, 0) \leq
            \dim_{[Z]}\Hilb^{\Gmult}(\mathbb{A}^n) + \dim_{\KK} \Hom_{R}(I, R/I)_{>0}$,
        \item\emph{tangent-to-base-and-fiber estimate} $\dim_{[Z]}
            \Hilb^+_k(\mathbb{A}^n, 0) \leq
            \dim_{\KK} \Hom_{R}(I, R/I)_0 + \dim_{[Z]}
            \pi^{-1}([Z])$,
        \item\emph{base-and-fiber estimate} $\dim_{[Z]}
            \Hilb^+_k(\mathbb{A}^n, 0) \leq
            \dim_{[Z]}\Hilb^{\Gmult}(\mathbb{A}^n) + \dim_{[Z]}
            \pi^{-1}([Z])$,
    \end{enumerate}

    In most of the remaining cases, these observations are sufficient.
    In few, we need to resort to the following observation, which we call \emph{local
    constancy}. If $[Z]$ is $\Gmult$-fixed and a smooth point of
    $\Hilb^+_k(\mathbb{A}^n, 0)$ then the map $\pi$ is smooth at $[Z]$ and so is
    $\Hilb_k^{\Gmult}(\mathbb{A}^n)$. In particular $\dim_{\KK}\Hom_{R}(I,
    R/I)_{>0}$ and $\dim_{\KK}\Hom_{R}(I, R/I)_0$ are constant in a
    neighbourhood of $[Z]$.
    This is used to rule out the cases where the tangent space has a surplus
    of exactly one with respect to the expected dimension: if
    the dimension of the stratum is indeed higher than expected, then the point has to be smooth and
    we can apply the above observation and find a witness of non-constancy.

    \subsection{Proof ideas II: explicit computations}\label{ssec:ideas2}
    Finally, there are much fewer but still quite a few cases where $\dim_{\KK}\Hom_R(I, R/I)_{\geq 0}$ is
    higher than $(n-1)(k-1)+1$ and for these we need to peek deeper into the
    structure of the stratum.
    We can rephrase the
    Diagram~\ref{eq:maindiagram} more concretely
    on the level of closed points as
    \begin{equation}\label{eq:mainconcrete}
        \begin{tikzcd}
            \bigsqcup_{H: \sum H = k}\left\{ I \subset R \ |\ \rad(I) =
                (x_1, \ldots ,x_n), H_{R/I} = H \right\}
                \ar[r, hook, "\mathrm{bij}"]\ar[d, "\gr"] & \left\{ I\subset R\ |\
                \rad(I) =
                (x_1, \ldots ,x_n),\ \dim_{\KK} R/I = k\right\}\\
                \bigsqcup_{H: \sum H = k}\left\{ I \subset R \ |\ I \mbox{
                graded, } H_{R/I} = H \right\}
        \end{tikzcd}
    \end{equation}
    where $R = \KK[x_1, \ldots ,x_n]$.
    Using this description, we see that  for a
    homogeneous ideal $I$, the fiber of $\pi$ over $[I]$ consists
    of all ideals having $I$ as the initial ideal, hence it is possible to
    describe the fiber explicitly. It is not quite true that such a description is
    straightforward, and
    in fact, it is not, especially since in many cases the fiber is very
    singular. To facilitate it we resort to a number of tricks and
    classification theorems that are too technical to describe here.
    We attempt to present only one, which can be called \emph{vertex removal}.

    Typically, the graded stratum
    $\HilbFuncGr{H}{n}$ contains small substrata consisting of highly singular
    points.
    The whole torus $\Gmult^n$ acts on $\HilbFunc{H}{n}$ and
    $\HilbFuncGr{H}{n}$, so we can fix a one-dimensional torus $T\subset\Gmult^n$
    and consider the induced decompositions.
    An example of such is the left vertex in Figure~\ref{fig:BBdecomposition}.
    The trick is to choose the weights of this other $T$ appropriately, so
    that near
    the vertex the action is divergent, as on the right part of
    Figure~\ref{fig:BBdecomposition}.
    Then the dimension of the non-vertex part can be bounded by using tangent
    spaces, while the dimension of the vertex part can be bounded directly,
    since this part becomes ``small'';  see
    Proposition~\ref{ref:14321final:prop} for an example of such an argument.

    \subsection*{Notation for the Hilbert schemes}
    In the table below, we summarize the various loci appearing in this
    article. See~\S\ref{sec:hilbertSchemePrelims} for details.

\begin{tabular}{c c}
    object & symbol\\
    \toprule
    The Hilbert scheme of $k$ points & $\Hilb_k(\mathbb{A}^n)$\\
    The punctual Hilbert scheme & $\Hilb_k(\mathbb{A}^n, 0)$\\
    The locus of algebras with socle dimension $\leq \tau$ & $\Hilb_k^{\tau}(\mathbb{A}^n, 0)$\\
    The locus of homogeneous ideals & $\Hilb_k^{\Gmult}(\mathbb{A}^n,0)$\\
    The locus of homogeneous ideals with fixed Hilbert function $H$ &
    $\Hilb_H^{\Gmult}(\mathbb{A}^n,0)$\\
    The \BBname{} decomposition of the punctual Hilbert scheme &
    $\Hilb_k^{+}(\mathbb{A}^n, 0)$\\
    The locus with fixed Hilbert function $H$ & $\HilbFunc{H}{n}$\\
    The locus of Gorenstein algebras with fixed Hilbert function $H$ &
    $\HilbFuncGor{H}{n}$
\end{tabular}
\goodbreak

    \section*{Acknowledgements}

    The authors would like to thank Jaros{\l{}}aw Buczy{\'n}ski and Tadeusz
    Januszkiewicz for sharing their idea about extending the notion of
    regularity to Grassmannians and for the permission to pursue this idea
    here. The authors are also very grateful to the referees for careful
    reading, catching many typos
    and suggesting many
    improvements to the original text.
    The first author is supported by Polish NCN grant
    2020/39/D/ST1/00132.

\section{Preliminaries}
In this section, we collect facts and notations regarding the
Hilbert scheme, the Hilbert function and Macaulay's inverse systems, used
throughout this article. As already stated in the introduction, throughout we work over an
algebraically closed field $\KK$ of characteristic zero.  All computations
below and in the introduction are available as an \emph{Macaulay2} ancillary
file for the arXiv version of this paper.

\subsection{Socle and Hilbert function}\label{sec:hilbertFunction}
\newcommand{\hatR}{\widehat{R}}
Let $\hatR=\KK[[x_1, \ldots ,x_n]]$ be the power series ring in $n$ variables with the maximal ideal $\mm=(x_1, \ldots,x_n)$.
Let $A=\hatR/I$ be a local Artin ring for an $\mm$-primary ideal $I\subset
\hatR$,
with the unique maximal ideal $\mathfrak{n}=\mm/I$ and residue field $\KK$.
Recall that an \emph{$\mm$-primary ideal} is just an ideal $I$ such that
$I\supset \mm^r$ for some $r$. Since for every $r$ we have
canonically
\begin{equation}\label{eq:ringchange}
    \frac{\hatR}{\mm^r}  \simeq
    \frac{R}{(x_1, \ldots,x_n)^r}
\end{equation}
we could view $A$ as a quotient of $R = \KK[x_1, \ldots ,x_n]$ by an ideal $I$
satisfying $I\supset (x_1, \ldots ,x_n)^r$ for some $r$.

\textit{The socle} $\Soc(A)$ of $A$ is the annihilator of the maximal ideal in $A$.
The socle is a $\KK$-vector space. \textit{The socle dimension} of $A$ is
defined by $\tau(A):=\dim_{\KK}\Soc(A)$. The ring $A$ is
\emph{Gorenstein} if
$\tau(A)=1$, see for example~\cite[Chapter~21]{EisView}.
The \emph{associated graded ring of $A$}, denoted $\gr(A)$, is the vector space
$\bigoplus_{i\geq
0}\mathfrak{n}^i/\mathfrak{n}^{i+1}$ with natural multiplication. When $A
\simeq
\KK[x_1, \ldots ,x_n]/I$, where $I\supset (x_1, \ldots ,x_n)^r$, then $\gr(A)
\simeq \KK[x_1, \ldots ,x_n]/\inn(I)$, where $\inn(I)$ is the
ideal
generated by the smallest degree forms of elements of $I$.
We have
$\tau(\gr A) \geq \tau(A)$ and typically strict inequality occurs.

\begin{definition}
    The \emph{Hilbert function} $H_A\colon\mathbb{N}\rightarrow \mathbb{N}$ of a local ring $A$
with a maximal ideal $\mathfrak{n}$ is defined by $H_A(i):=\dim_{\KK}
\mathfrak{n}^i/\mathfrak{n}^{i+1}$. It is the Hilbert function of its
associated graded ring.
\end{definition}

Let $s$ denote the largest integer such that $\mathfrak{n}^s\neq 0$, the
so-called \textit{socle degree} of $A$. The Hilbert function of $A$ can be
then represented by a vector $ H_A=(1,H_A(1),\ldots, H_A(s))$ or a series
$\sum_{i=0}^s H_A(i)T^i$. Moreover, since $\mathfrak{n}^{s}\subset \Soc(A)$, we
get $H_A(s)\leq \tau(A)$. We define the Hilbert function of an
$A$-module similarly.

\subsection{The Hilbert scheme of points and its
subloci}\label{sec:hilbertSchemePrelims}

    The Hilbert scheme $\Hilb_k(\mathbb{A}^n)$ of $k$ points on $\mathbb{A}^n$
    is an open subscheme of the projective scheme $\Hilb_k(\mathbb{P}^n)$.
    Numerous nice introductions to Hilbert scheme of points
    exist, for example~\cite[Chapters~5-6]{fantechi_et_al_fundamental_ag},
    \cite{Stromme_Hilbert, Bertin__punctual_Hilbert_schemes}. Here we only
    collect the facts which will be useful in the article.

    Let $R \simeq \KK[x_1, \ldots ,x_n]$. The \emph{degree} of a zero-dimensional subscheme $Z
    = \Spec(R/I)$ of
    $\mathbb{A}^n$ is $\deg(Z) := \dim_{\KK} (R/I)$.
    Closed points of $\Hilb_k(\mathbb{A}^n)$ correspond bijectively
    to zero-dimensional closed subschemes $Z \subseteq \mathbb{A}^n$ of degree
    $k$, or, in other words, to ideals $I\subseteq R$ with
    $\dim_{\KK} R/I = k$. We write $[Z]$ or $[I]$ for the point corresponding
    to a subscheme $Z$ or ideal $I$.
    The tangent space at a closed point $[\Spec(R/I)]\in
    \Hilb_k(\mathbb{A}^n)$ is given by $\Hom_{R}(I, R/I)$, see for
    example~\cite[Theorem~10.1]{Stromme_Hilbert}. The locus of $[Z]\in
    \Hilb_k(\mathbb{A}^n)$ with $Z$ smooth is open in the Hilbert scheme. Its
    closure is called the \emph{smoothable component}, we denote it by
    $\HilbSmk(\mathbb{A}^n)$.

    A zero-dimensional subscheme $Z$ is a finite disjoint union $Z_1\sqcup
    \ldots \sqcup Z_{e}$, where the summands are irreducible
    subschemes, so, as a set, every $Z_i$ is a singleton
    $p_i\in \mathbb{A}^n$. The \emph{support} of $Z$ is a formal sum
    $\sum_{i=1}^e \deg(Z_i)p_i$.
Let $\Sym_k(\mathbb{A}^n) = (\mathbb{A}^n)^k \goodquotient \mathbb{S}_k$ be the $k$-fold symmetric product of
$\mathbb{A}^n$. The Hilbert-Chow morphism $\rho_k\colon
\Hilb_k(\mathbb{A}^n)\to \Sym_k(\mathbb{A}^n)$ sends a zero-dimensional scheme
$Z\subset \mathbb{A}^n$ to its
support~\cite[Theorem~2.16]{Bertin__punctual_Hilbert_schemes}. It
extends to a morphism $\bar{\rho}_k\colon \Hilb_k(\mathbb{P}^n)\to
\Sym_k(\mathbb{P}^n)$ and we have a cartesian diagram
\[
    \begin{tikzcd}
        \Hilb_k(\mathbb{A}^n)\ar[r, "\rho_k"]\ar[d, hook, "open"] &
        \Sym_k(\mathbb{A}^n)\ar[d, hook, "open"]\\
        \Hilb_k(\mathbb{P}^n) \ar[r, "\bar{\rho}_k"] & \Sym_k(\mathbb{P}^n)
    \end{tikzcd}
\]
The \emph{punctual Hilbert scheme} is defined as $\Hilb_k(\mathbb{A}^n, 0) :=
\rho_k^{-1}(d[0]) = \bar{\rho}_k^{-1}(d[0])$, where $0\in \mathbb{A}^n$ is the
origin. As a topological space, the punctual Hilbert scheme consists of subschemes supported only at the
origin. The description by $\bar{\rho}_k$ implies that this scheme is
projective.
We obtain the corresponding locus
\[
    \HilbSmk(\mathbb{A}^n, 0) := \Hilb_{k}(\mathbb{A}^n, 0)\cap
    \HilbSmk(\mathbb{A}^n).
\]

The invariants discussed in~\S\ref{sec:hilbertFunction} give rise to loci in
the punctual Hilbert scheme which we now formally introduce.
For $[I]\in \Hilb_k(\mathbb{A}^n, 0)$ the socle dimension of $R/I$ is the
minimal number of generators of the canonical module $\omega_{R/I}$, so the
sublocus $\Hilb_k^{\tau}(\mathbb{A}^n, 0)\subseteq \Hilb_k(\mathbb{A}^n,0)$
that consists of $[I]$ with
$\tau(R/I)\leq \tau$ is open.
For every $i$, the dimension
$\dim_{\KK} R/(I+R_+^i)$ is upper-semicontinuous as $[I]$ varies in the punctual
Hilbert scheme, hence for a given $H\colon \mathbb{N}\to \mathbb{N}$ the locus
of $[I]$ with $H_{R/I} = H$ is locally closed. We denote it by
$\HilbFuncTheta{H}{n}$.

For generalities on the \BBname{} decomposition $\Hilb_k^+(\mathbb{A}^n)$ of the Hilbert scheme, we
refer the reader to~\cite{Jelisiejew__Elementary}. The \BBname{} decomposition
$\Hilb_k^+(\mathbb{A}^n, 0)$ of the punctual Hilbert scheme can be formally
defined using the machinery of~\cite{jelisiejew_sienkiewicz__BB}.
It follows from~\cite[Proposition~3.3]{Jelisiejew__Elementary} that the
inclusion $\Hilb_k^+(\mathbb{A}^n, 0)\into \Hilb_k^+(\mathbb{A}^n)$ is an
isomorphism, but we will not use this fact in the present article. The
restriction of the map $\theta$ from the introduction is a map
$\Hilb^+_k(\mathbb{A}^n, 0)\to \Hilb_k(\mathbb{A}^n, 0)$ which is bijective on
points (but not an isomorphism!). We identify the closed points of those two
schemes.
There is a morphism $\pi\colon \Hilb_k^+(\mathbb{A}^n, 0)\to
\Hilb_k^{\Gmult}(\mathbb{A}^n, 0)$ which on the level of points sends $[I]$ to
the initial ideal $[\inn(I)]$.
We define $\HilbFunc{H}{n}$ as $\pi^{-1}(\HilbFuncGr{H}{n})$. The forgetful
map $\Hilb^+_k(\mathbb{A}^n, 0)\to \Hilb_k(\mathbb{A}^n, 0)$ restricts to a
map $\HilbFunc{H}{n}\to \HilbFuncTheta{H}{n}$ bijective on points and so
\[
    \dim \HilbFunc{H}{n} = \dim \HilbFuncTheta{H}{n};
\]
in fact those two schemes can be thought of as two scheme structures on the
same set of closed points.
We have a much better control on the tangent space of
$\HilbFunc{H}{n}$, so we will mostly consider this one.
We define the open sublocus
\[
    \HilbFuncGor{H}{n} := \HilbFunc{H}{n}\cap \Hilb_k^{1}(\mathbb{A}^n, 0).
\]
The scheme $\HilbFuncGr{H}{n}$ is a union of connected components of
$\Hilb^{\Gmult}_k(\mathbb{A}^n, 0)$, so the scheme $\HilbFunc{H}{n}$
is a
union of connected components of $\Hilb_k^+(\mathbb{A}^n, 0)$.
We call this union the \emph{stratum} corresponding to $H$, or, when an
ideal $I$ with $H_{R/I} = H$ is given, the \emph{stratum} of $I$. (We
believe that $\HilbFuncGr{H}{n}$ is connected, so the stratum is connected as
well. However, this result seems
to be not present explicitly in the literature and it is not important for the
current article.)

Let $[I]\in \Hilb^{\Gmult}_k(\mathbb{A}^n)$. Since
$I$ is a homogeneous ideal, the tangent space at $[I]\in \Hilb_k(\mathbb{A}^n)$ has a structure of a graded
$R$-module $\Hom_R(I,R/I)=\bigoplus_{d\in \mathbb{Z}}\Hom_R(I,R/I)_d$,
with $\Hom_R(I,R/I)_d$
    consisting of all homomorphisms $\varphi\colon I\to R/I$ of $R$-modules such that $\varphi(I_i) \subset
    (R/I)_{i+d}$ for all $i$. We set
    \[
        \Hom_R(I,R/I)_{\geq 0} := \bigoplus_{i\geq 0} \Hom_R(I,R/I)_i\quad
        \mbox{and}\quad \Hom_R(I,R/I)_{>0} := \bigoplus_{i>0}
        \Hom_R(I,R/I)_i.
    \]
By~\cite[Theorem~4.2]{Jelisiejew__Elementary} the tangent space at $[I]\in
\Hilb_k^+(\mathbb{A}^n,0)$ is isomorphic to $\Hom_{R}(I, R/I)_{\geq 0}$. By a similar
argument, the tangent space at $[I]$ to $\pi^{-1}([I])$ is $\Hom_{R}(I,
R/I)_{>0}$.

Although the tangent space $\Hom_{R}(I, R/I)$ can be computed by hand, we usually use the
    computer algebra system \textit{Macaulay2} to compute it quickly in
    explicit cases, most of them related to monomial ideals. See
    Proposition~\ref{ref:tangentSpacesToStrata:prop} for an example of
    such a computation.

\subsection{Macaulay's Inverse Systems}
Let $S=\KK[y_1,\ldots,y_n]$ be the polynomial ring with $n$ variables which we
view as a vector space.
Recall the rings $R = \KK[x_1, \ldots ,x_n]\subset \hatR = \KK[[x_1, \ldots
,x_n]]$.
The
ring $\hatR$ acts on $S$ by the partial derivation map $\circ\colon\hatR\times S \rightarrow S$
defined as follows:
\[
    x^\alpha\circ y^\beta:=\begin{cases}
        \binom{\beta}{\alpha} y^{\beta-\alpha} & \beta\geq  \alpha,\\
        0 & \mbox{otherwise},
    \end{cases}
\]
where $\alpha=(\alpha_1,\ldots,\alpha_n)$ and $\beta=(\beta_1,\ldots,\beta_n)$
are vectors in $\mathbb{N}^n$, $\binom{\beta}{\alpha} = \prod_{i=1}^n
\binom{\beta_i}{\alpha_i}$, and $\beta\geq \alpha$ if  $\beta_i\geq
\alpha_i$ for all $1\leq i\leq n$. With this action, $S$ can be viewed as an
$\hatR$-module. In fact, the partial derivation map is up to scalars equivalent to the
action of $\hatR$ on $S$ by contraction, see for
example~\cite[Appendix~A]{iakanev}, because we are over a field of
characteristic zero.
We will mostly work with the polynomial ring $R$ rather than the whole
$\hatR$, thus below we restrict to the $R$-action on $S$. This is mostly a
formal choice.
\begin{definition}
Let $I\subset R$ be an $(x_1, \ldots ,x_n)$-primary ideal of $R$. The
Macaulay \emph{inverse system of $I$} is
\[
    I^{\perp}=\{ g\in S: \ I\circ g= 0\}.
\]
This is an $R$-submodule of $S$. Its generators are called
\emph{dual generators} of $I$.
Given a subset $E \subset
S$, the annihilator of $E$ is the ideal
\[
    \Ann_{\hatR}(E)=\{f\in R: \ f\circ E=0\},
\]
which is called the \emph{apolar ideal} of $E$.
\end{definition}

Note that $\Ann_{R}(E) = \Ann_{R}(M)$, where $M = R\circ E$ is the
$R$-submodule
of $S$ generated by $E$. If $I\subseteq R$ is a homogeneous ideal, then
$I^{\perp}$ is spanned by homogeneous polynomials, and if $E$ is spanned by
homogeneous elements, then $\Ann_R(E)\subseteq R$ is homogeneous. Also, if $I\supset (x_1, \ldots ,x_n)^r
$,  then $I^{\perp} \subset S_{<r}$.
We will often abbreviate $\Ann_R(-)$ to $\Ann(-)$ when the ring is clear from
the context.

The above constructions are
justified by the following theorem.
\begin{theorem}[Macaulay's duality]
    For every $(x_1, \ldots ,x_n)$-primary ideal $I$ of $R$ and finitely generated
    $R$-submodule $M\subset S$ we have $\Ann_{R}(I^{\perp}) = I$ and
    $(\Ann_{R}(M))^{\perp} = M$. In this way, the operations $\perp$ and
    $\Ann_{R}(-)$ give
    an inclusion-reversing bijection between finitely generated $R$-submodules of
    $S$ and $(x_1, \ldots ,x_n)$-primary ideals of $R$. Moreover, the $R$-module
    $I^\perp$ is minimally generated by $\tau(R/I)$ elements.
    A zero-dimensional $A=R/I$ is Gorenstein of socle degree $s$ if and only
    if $I^\perp$ is a principal $R$-module generated by a polynomial of degree $s$.
\end{theorem}

\begin{proposition}[parameterizing by polynomials vs by the Hilbert
    scheme]\label{ref:polynomialsAndGorensteinAlgebras:prop}
    For a Hilbert function $H$ with $k = \sum H$ consider the set $\mathcal{L}_H$ of all polynomials $f$ such
    that $R/\Ann_{R}(f)$ has Hilbert function $H$; this is a locally closed
    subset of $S_{<k}$. Then there is a surjective map
    $\mathcal{L}_H\to \HilbFuncGor{H}{n}$ whose fibers are
    $k$-dimensional. In particular, we have
    \[
        \dim \HilbFuncGor{H}{n} = \dim \mathcal{L}_H - k.
    \]
\end{proposition}
\begin{proof}
    For a polynomial $f\in S_{<k}$ to fix the Hilbert function of
    $R/\Ann_{R}(f)$
    is the same as fixing the dimensions of the spaces $\mm^i\circ f$ for every
    $i$. This shows that $\mathcal{L}_H$ is locally closed.
    The surjectivity and the claim about fiber dimensions are classical, see
    for example~\cite {emsalem, iarrobino_associated_graded}.
\end{proof}

\subsection{Lower bounds on dimension of components of $\HilbSmk(\mathbb{A}^n,
0)$}\label{sec:irreducibilitypreliminaries}

In this section, we prove that every component of $\HilbSmk(\mathbb{A}^n, 0)$
has \emph{at least} the excepted dimension $(n-1)(k-1)$. We being with a
general observation.

\begin{proposition}[weak purity of the total branch
    locus]\label{ref:purity:prop}
    Let $f\colon X\to S$ be a finite flat degree $d$ morphism of schemes locally of
    finite type over $\KK$. Assume $f$ is generically \'etale. Consider the locus
    \[
        \mathcal{Z} := \left\{ s\in S\ |\ \mbox{ the geometric fiber } f^{-1}(s)\mbox{ has only one point} \right\}
    \]
    This is a closed subset of $S$ and for every $z\in \mathcal{Z}$ we have
    $\dim_{z} \mathcal{Z} \geq \dim_z S - (d-1)$.
\end{proposition}
\begin{proof}
    \def\pp{\mathfrak{p}}%
    Closedness of $\mathcal{Z}$ is well-known, for example~\cite[\href{https://stacks.math.columbia.edu/tag/0BUI}{Tag 0BUI}]{stacks_project}.
    The statement on the dimension is local around $z$, so we may replace $S$ by a component
    through $z$ which has dimension $\dim_z S$. Next, we reduce to the
    situation when $f\colon X\to S$ has $d$ sections $s_1, \ldots ,s_d$ such
    that $X = \bigcup s_i(S)$.
    Assume we have already $p$ sections $s_1, \ldots ,s_p$, where $p$ is
    possibly zero.
    Suppose $X \neq \bigcup s_i(S)$, take any irreducible component $S'$ of $X$ other than
    $s_1(S)$, \ldots ,$s_p(S)$ and view it with reduced
    scheme structure. Consider the
    base change
    \[
        \begin{tikzcd}
            X' := X\times_S S' \ar[r]\ar[d, "f'"] & X\ar[d,
            "f"]\\
            S' \ar[r] & S
        \end{tikzcd}
    \]
    The map $f'$ is finite flat degree $d$ by base change. Since $f$ is flat,
    it is torsion free, so every component of $X$ dominates $S$. Moreover, $f$
    is closed as a finite map. Thus $S'\to S$ is onto and so $f'$ is
    generically \'etale.
    This map has $p+1$ sections: $p$ sections which are pullbacks of $s_i$ and an
    additional section $i\times \id$ that comes from $i\colon S'\to X$. The
    intersections of every two of these sections with a fiber over a general
    point of $S'$ are disjoint, hence indeed we obtained $p+1$ disjoint
    sections (this in particular shows that one could not obtain $d+1$
    sections, so the procedure ends exactly for $p=d$).
    By replacing $S$ as above, we assume $f\colon X\to S$ is such that $X =
    \bigcup_{i=1}^d X_i$, where $f|_{X_i}\colon X_i\to S$ are isomorphisms.
    We finally replace $S$ by its normalization (which is a finite surjective
    map since $S$ was a finite type scheme) and shrink it so that $S =
    \Spec(A)$ is affine and $X = \Spec(B)$, where $B$ is a free $A$-algebra.

    We observe now that $X$ is connected in codimension one, which by
    definition means that $X\setminus V$ is connected for every closed $V$ with
    $\dim V \leq \dim S - 2$. Indeed, suppose not, then there exist nontrivial
    idempotents in $H^0(X\setminus V, \OO_{X})$. Since $A$ is normal, we have
    \[
        A = \bigcap\left\{ A_{\pp}\ |\ \pp\in \Spec(A), \dim(A/\pp) =
        \dim(A) - 1\right\}.
    \]
    Since $B$ is a free $A$-module, we have
    \[
        B = \bigcap\left\{ B_{\pp}\ |\ \pp\in \Spec(A), \dim(A/\pp) =
    \dim(A) - 1\right\},
    \]
    where $B_{\pp} = (A\setminus \pp)^{-1}B$.
    Therefore, the idempotents above uniquely extend to global sections
    $H^0(X, \OO_X)$ which are idempotents themselves; so $X$ is disconnected.
    But this is absurd since the unique preimage of $z$ lies in every
    irreducible component $X_i$. This concludes the proof that $X\setminus V$
    is connected.

    The locus where $f\colon X\to S$ is not \'etale is, locally near $z$,
    given by a single
    equation~\cite[\href{https://stacks.math.columbia.edu/tag/0BVH}{Tag~0BVH}]{stacks_project}. This locus coincides with the locus of $s\in
    S$ such that in $f^{-1}(s)$ two of the components $X_1$, \ldots , $X_d$
    coincide. Consider the graph with vertices $X_i$ and an edge joining $X_i$
    and $X_j$ iff $\dim(X_i\cap X_j) = d-1$. If this graph is disconnected,
    then $X$ is not connected in codimension one. Hence this graph is
    connected and we may choose a spanning tree $e_1 = (X_{i_1}, X_{j_1}),
    \ldots, e_{d-1} = (X_{i_{d-1}}, X_{j_{d-1}})$. The image $S_{m}$ of
    $X_{i_m} \cap
    X_{j_m}$ is a codimension one subscheme of $S$ that contains $z$, so
    the intersection $S_1 \cap S_2 \cap  \ldots \cap S_{d-1}$ is near $z$ a
    subscheme of codimension at most $d-1$, see for example~\cite[Chapter 10,
    10.2, 10.5, 10.6]{EisView}. For every point $z'$ of this
    intersection consider the fiber $F = f^{-1}(z')$. By assumption, the
    intersection $F\cap X_{i_m}\cap X_{j_m}$ is nonempty for every $m$. But
    $F\cap X_{i_m} = F\cap s_i(S)$ is (as a scheme!) a point, and similarly for $F\cap
    X_{j_m}$. This shows that $F\cap X_{i_m} = F\cap X_{j_m}$ for every $m$.
    Since we iterate over a spanning tree of a graph with vertices $X_1$,
    \ldots, $X_d$, we deduce that topologically $|F|\cap |X| = |F| \cap
    \bigcup_{i} |X_i| = \bigcup_{i} |F|\cap |X_i|$ is a point. Therefore,
    $z'\in \mathcal{Z}$ and so $S_1 \cap S_2 \cap  \ldots \cap S_{d-1}\subset
    \mathcal{Z}$, which concludes the proof.
\end{proof}

As a Corollary, we get the following result, whose latter part was proven
classically by different methods in~\cite[Theorem~3.5]{Gaffney}.
\begin{proposition}\label{ref:componentsLowerBound:prop}
    The dimension of
    $\HilbSmk(\mathbb{A}^n, 0)$ at every its point is at least $(n-1)(k-1)$.
\end{proposition}
\begin{proof}
    By Proposition~\ref{ref:purity:prop} applied to the universal family of
    $\HilbSmk(\mathbb{A}^n)$ we get that the locus $\mathcal{Z} \subset
    \HilbSmk(\mathbb{A}^n)$ parameterizing schemes supported at a single point
    has codimension at most $k-1$. This locus is fibered over
    $\mathbb{A}^n$ by sending the scheme to its support (formally, this
    map comes from restricting the Hilbert-Chow morphism). The locus
    $\HilbSmk(\mathbb{A}^n, 0)$ is a fiber of $\mathcal{Z}\to \mathbb{A}^n$,
    hence it has
    codimension at most $k-1+n$.
\end{proof}

Iarrobino~\cite[p.310]{iarrobino_10years} asked whether the lower bound
$(n-1)(k-1)$ from Proposition~\ref{ref:componentsLowerBound:prop} holds for
every component of $\Hilb_k(\mathbb{A}^n, 0)$. Recently, Satriano and
Staal~\cite{Satriano_Staal}
gave a negative answer already for $n=4$.

\section{Irreducibility of punctual Hilbert schemes on threefolds for $k\leq
11$}\label{Irreducibility}

In this subsection, we prove the irreducibility of $\Hilb_k(\mathbb{A}^3, 0)$ for
$k\leq11$.
Recall from \S\ref{sec:hilbertSchemePrelims} that the stratum
$\HilbFunc{H}{3}$ of an
ideal $I$ is the union of some connected components of
$\Hilb_k^+(\mathbb{A}^3, 0)$; the closed points of this union are ideals $I'$
with $H_{R/I'} = H_{R/I}$ and the tangent space to the stratum at a
$\Gmult$-fixed $[I]$ is given by
$\Hom_{R}(I, R/I)_{\geq 0}$.
As explained in the introduction, having
Proposition~\ref{ref:componentsLowerBound:prop}, the main difficulty is in
proving that the dimension of all but one strata is smaller than
expected. We do it
by bounding the dimension of their tangent spaces.
\begin{proposition}\label{ref:tangentSpacesToStrata:prop}
    Consider Borel-fixed ideals $I \subset \KK[x_1, x_2, x_3]$ such that
    $\dim_{\KK} \KK[x_1, x_2, x_3]/I \leq 11$. Then
    \[
        D(I) := \dim_{\KK} \left(T_{[I]} \Hilb_k(\mathbb{A}^3)\right)_{\geq 0} - 2(k-1)
    \]
    is negative for all of them with the exception of the following ideals:
    \begin{enumerate}[label=(\roman*)]
        \item  $I = (x_1, x_2, x_3^k)\mbox{ where } D(I) = 0$,
        \item\label{ex:1331} $I =
            (x_2^2,x_1x_2,x_1^2,x_2x_3^2,x_1x_3^2,x_3^4)\mbox{ where } D(I) =
            0$ and $H_{R/I} = (1,3,3,1)$,
        \item\label{ex:13311} $I = (x_2^2,x_1x_2,x_1^2,x_2x_3^2,x_1x_3^2,x_3^5)\mbox{ where } D(I) =
            0$ and $H_{R/I} = (1,3,3,1,1)$, 
        \item\label{ex:13321a} $I = (x_1x_3,x_1x_2,x_1^2,x_2^2x_3,x_2^3,x_2x_3^3,x_3^5)\mbox{ where }
            D(I) = 0$ and $H_{R/I} = (1,3,3,2,1)$,
        \item\label{ex:133111} $I = (x_2^2,x_1x_2,x_1^2,x_2x_3^2,x_1x_3^2,x_3^6) \mbox{ where } D(I) =
            0$ and $H_{R/I} = (1,3,3,1,1,1)$,
        \item\label{ex:13321b} $I = (x_2^2,x_1x_2,x_1^2,x_1x_3^2,x_2x_3^3,x_3^5) \mbox{ where } D(I)
            = 2$ and $H_{R/I} = (1,3,3,2,1)$,
        \item\label{ex:133211a} $I = (x_1x_3,x_1x_2,x_1^2,x_2^2x_3,x_2^3,x_2x_3^3,x_3^6) \mbox{ where }
            D(I) = 0$ and $H_{R/I} = (1,3,3,2,1,1)$,
        \item\label{ex:1331111} $I = (x_2^2,x_1x_2,x_1^2,x_2x_3^2,x_1x_3^2,x_3^7) \mbox{ where } D(I) =
            0$ and $H_{R/I} = (1,3,3,1,1,1,1)$,
        \item\label{ex:133211b} $I = (x_2^2,x_1x_2,x_1^2,x_1x_3^2,x_2x_3^3,x_3^6) \mbox{ where } D(I) =
            2$ and $H_{R/I} = (1,3,3,2,1,1)$,
        \item\label{ex:13331} $I = (x_2^2,x_1x_2,x_1^2,x_2x_3^3,x_1x_3^3,x_3^5) \mbox{ where } D(I) =
            0$ and $H_{R/I} = (1,3,3,3,1)$,
        \item\label{ex:13421} $I = (x_1x_2,x_1^2,x_1x_3^2,x_2^2x_3,x_2^3,x_2x_3^3,x_3^5)\mbox{ where
            } D(I) = 1$ and $H_{R/I} = (1,3,4,2,1)$.
    \end{enumerate}
\end{proposition}
\begin{proof}
    Direct \emph{Macaulay2} check using the package
    \emph{StronglyStableIdeals}~\cite{StronglyStableIdeals}. See the ancillary
    file for the arXiv version of this article.
\end{proof}

Before we deal with bounding the dimensions of respective strata, we need to
consider separately two ``sporadic'' cases, where $D(I)$ is too positive
to deal with it smoothly. For the first of the two, we consider not only
$\mathbb{A}^3$ but, more
generally, $\mathbb{A}^n$, as we will use also the case $n=4$ below, in
Proposition~\ref{ref:14321final:prop}.
\begin{proposition}\label{ref:1m321:prop}
Let $H=(1,n,3,2,1)$. The Hilbert scheme $\HilbFuncGr{H}{n}$ set-theoretically
consists of the following loci:
\begin{enumerate}[label=(\arabic*)]
    \item\label{it:1m321easy} ideals dually generated by $\ell_1^4, \ell_2^3,q$, where
        $\ell_1,\ell_2$ are two independent linear forms, and $q$ is a quadric
        independent from $\ell_1^2,\ell_2^2$. This is a locus parameterized by an open
        subset of $\Gr(1,n) \times \Gr(1,n) \times \Gr(1,\binom{n+1}{2} - 2)$,
        whence it has
        dimension $2n+\binom{n+1}{2}-5$.
    \item\label{it:1m321easy2} Ideals dually generated by $Q$ and $q$, where $Q$
        lies in $\Sym^4(V)$ for a two-dimensional linear subspace $V\subset
        S_1$ and is such that the space of second-order partial derivatives of $Q$
        is two-dimensional, while $q$
        is a
        quadric linearly independent from the second order partial derivatives of
        $Q$.
        The parameter space for $Q$ is an open subset of
        a divisor in a $\mathbb{P}(\Sym^4 \KK^2)$-bundle over $\Gr(2, n)$,
        while the parameter space for $q$ is an open subset of $\Gr(1,
        \binom{n+1}{2} - 2)$. Summing up, the locus
        is irreducible and has dimension
        $2n+\binom{n+1}{2}-4$.
    \item\label{it:1m321last} Ideals dually generated by all linear forms and
        a
        $Q$, where $Q$ is a general quartic in $\Sym^4(V)$ for a
        two-dimensional linear subspace $V\subset S_1$. This locus is
        parameterized by an open subset of a $\mathbb{P}(\Sym^4
        \KK^2)$-bundle over $\Gr(2,n)$ and hence has dimension $2n$.
    \item\label{it:1m321subtle} Ideals dually generated by a perfect quartic $\ell^4$ and a non-perfect cubic
        $c$ whose first order partial derivatives together with $\ell^2$
        span a $3$-dimensional subspace. This gives a locus of dimension
        at most $3n-2$.
\end{enumerate}
\end{proposition}
\noindent In the case~\ref{it:1m321subtle} we do not have a parameterization,
        but the estimate above is enough for our purposes.
\begin{proof}
 Let $I$ be the homogeneous ideal of an algebra with the given Hilbert
 function.  Since $I^\perp$ is generated (non-minimally) by a quartic, two
 cubics and three quadrics as an $R$-module, the
 possibilities~\ref{it:1m321easy}-\ref{it:1m321last} and the corresponding
 dimension counts are obvious. We only prove the bound on the dimension of the
 last locus.
 Let $\ell^4$ be a perfect quartic and let $c$ be a cubic
 whose partial derivatives together with $\ell^2$ generate a $3$-dimensional
 vector space.
 If $c$ does not depend essentially on three variables, then the
 upper bound $n-1 + \dim \Gr(2, n) + 3 = 3n-2$ follows immediately, so we suppose
 it does. Let $c\in \KK[W]$, for $W\subset S_1$
 three-dimensional.
 The choice of $c$ is determined by a choice of a linear form $\alpha\in
 W^{\vee}$
 such that $\alpha \circ c = \ell^2$ and a choice of an element in $\Sym^3V$,
 where $V$ is the two-dimensional space of linear forms in $W$ annihilated by $\alpha$.
 At most, we obtain a $\dim \Gr(3, n) + \dim \Gr(1, W) + \dim W^{\vee} + \dim
 \KK[V]_3 = 3n$ dimensional choice. This parameterization is not one-to-one,
 as adding $\ell^3$ to $c$ and/or multiplying it by a nonzero scalar does not
 change the system. Therefore, the final estimate is $3n-2$, as claimed.
 \end{proof}
\begin{remark}\label{ref:1m3211:rem}
By the very same argument as in Proposition~\ref{ref:1m321:prop}, for
$H=(1,n,3,2,1,1, \ldots, 1)$, with at least two trailing ``1'', the Hilbert scheme $\HilbFuncGr{H}{n}$ consists of two loci, parameterizing ideals analogous to the ideals
in~\ref{it:1m321easy} and~\ref{it:1m321subtle},
and with the same dimension estimates, respectively.
\end{remark}

\begin{proposition}\label{ref:13421:prop}
Let $H=(1,3,4,2,1)$. The Hilbert scheme $\HilbFuncGr{H}{3}$ set-theoretically consists of the following loci:
\begin{enumerate}[label=(\arabic*)]
    \item ideals dually generated by $\ell_1^4,\ell_2^3$ for two independent
        $\ell_1,\ell_2\in S_1$, and two quadrics $q_1,q_2$ independent from
        $\ell_1^2,\ell_2^2$. This is a locus parametrized by an open subset of
        $\Gr(1,3)\times \Gr(1,3)\times \Gr(2,4)$ whence it has dimension $8$.
    \item Ideals dually generated by a perfect quartic $\ell^4$ and a general
        cubic. The cubic can be taken modulo the $\ell^3$ term, so the locus is
        parameterized by an open subset of $\Gr(1,3)\times \Gr(1,9)$, so it has dimension $10$.
    \item Ideals dually generated by $Q$ and $q_1$, $q_2$, where $Q$
        lies in $\Sym^4(V)$ for a two-dimensional linear subspace $V\subset
        S_1$ and is such that the space of second-order partial derivatives of $Q$
        is two-dimensional, while $q_1$, $q_2$
        are
        quadrics linearly independent from the second order partial derivatives of
        $Q$.
        This locus is an open subset of a divisor in a $\mathbb{P}(\Sym^4
        \KK^2)$-bundle over $\Gr(2, 3) \times
         \Gr(2, 4)$, whence it has dimension
        $(2 + 4 + 4) -1 = 9$.
    \item Ideals dually generated by $Q$ and $q$, where $Q$
        is an element in $\Sym^4(V)$ for a two-dimensional linear subspace $V\subset
        S_1$ and is such that the space of second-order partial derivatives of $Q$
        is three-dimensional, while $q$ is
        a quadric linearly independent from the second order partial derivatives of
        $Q$. This locus is parametrized by an open subset of a
        $\mathbb{P}(\Sym^4 \KK^2)$-bundle over $\Gr(2,3)\times \Gr(1, 3)$ and hence of dimension $8$.
\item \label{it:13421subt} Ideals dually generated by $\ell^4$, $c$, $q$,
    where $\ell$ is a linear form, $c$ is a cubic in $S$ with first-order
    partial derivatives are linearly dependent with $\ell^2$, and a quadric $q$
independent from partial derivatives of $c$ and $\ell^2$. This gives a
locus of dimension at most $9$.
\end{enumerate}
\end{proposition}
\begin{proof}
We only note that the bound $9$ in \ref{it:13421subt}
follows from the bound $7$ in Proposition
\ref{ref:1m321:prop}\ref{it:1m321subtle} increased by $2 = \dim \Gr(1, 3)$ to account for the
choice of a quadric $q$ modulo the partials of $c$ and $\ell^4$.
\end{proof}

\begin{proposition}\label{ref:smallExceptionalStrata:prop}
    The stratum of $(x_1, x_2, x_3^k)$ is $2(k-1)$-dimensional, while
    the strata for all other exceptional ideals from
    Proposition~\ref{ref:tangentSpacesToStrata:prop} have dimension strictly less
    than $2(\deg(I) - 1)$. Moreover, the point $(x_1, x_2, x_3^k)$ is a smooth
    point of its stratum.
\end{proposition}
\begin{proof}
    Consider first the curvilinear ideal $I = (x_1, x_2,
    x_3^k)$. We directly see that the dimension
    of the tangent space $\Hom_R(I, R/I)_{\geq 0}$ to its stratum is $2(k-1)$. Moreover, for every $f_1, f_2\in
    (tk[t])_{\leq k-1}$ the ideal $(x_1 - f_1(x_3), x_2 - f_2(x_3), x_3^k)$
    lies in the stratum of $I$. Hence this stratum has dimension at least
    $2(k-1)$ near $[I]$. Together with the tangent space dimension, this implies that
    $[I]$ is a smooth point of its $2(k-1)$-dimensional stratum.

    Consider now the remaining ideals $I$. First, pick any of them with
    $D(I) = 0$. Suppose that its stratum is $2(k-1)$-dimensional. Then $D(I) = 0$
    implies that $[I]$ is a smooth point of its stratum.
    Suppose $I'$ is a homogeneous ideal and suppose there is a one-dimensional
    torus $T$ action on
    $\AA^3$ such that in the induced action on $\Hilb_k(\mathbb{A}^3)$
    the point $[I']$ converges to $[I]$. The
    curve $[I_t]$ degenerating $[I']$ to $[I]$ lies in the $\Gmult$-fixed locus and
    in particular in the stratum of $[I]$. Since $[I]$ is a smooth point
    of the stratum, the smooth locus of the stratum intersects the curve
    $[I_t]$ in an open, $\Gmult$-stable subset that
    contains the unique $\Gmult$-fixed point of the curve; thus the smooth locus contains
    the whole curve $[I_t]$. In particular, the point $[I']$ is smooth in this
    stratum,
    hence
    \[
        \left(T_{[I']} \Hilb_k(\mathbb{A}^3)\right)_{\geq 0} = 2(k-1).
    \]
    By semicontinuity of each graded piece of the tangent space, we even
    obtain
    \[
        \left(T_{[I]} \Hilb_k(\mathbb{A}^3)\right)_{i} =\left(T_{[I']}
        \Hilb_k(\mathbb{A}^3)\right)_{i} \quad \mbox{for all}\quad i\geq 0.
    \]
    It is enough to exhibit an ideal $I'$ violating this last condition. Below
    we give a list of those, in each case pointing out the differing
    degree.
    In all cases, we fix a $T$-action $t\cdot (y_1, y_2, y_3) =
    (t^{a_1}y_1, t^{a_2}y_2, t^{a_3}y_3)$ with $a_1 \gg a_2 \gg
    a_3$ and consider the limit at $t\to 0$, so the initial form of a
    homogeneous polynomial is its
    lex-largest monomial in $S$.
    \begin{itemize}
        \item[\ref{ex:1331}, \ref{ex:13311}, \ref{ex:133111},
            \ref{ex:1331111}] consider the apolar ideal of $y_1^2+y_1y_3, y_2y_3,
            y_3^{3+i}$, degree $i$ for $i=0,1,2,3$, respectively.
        \item[\ref{ex:13321a}, \ref{ex:133211a}] consider the apolar ideal of $y_1^2+y_2^2,
            y_2y_3^2, y_3^{4+i}$, degree $i$ for $i=0,1$, respectively.
        \item[\ref{ex:13331}] consider the apolar ideal of $y_1y_3^2 +
            y_2^2y_3, y_2y_3^2, y_3^4$, degree $1$.
    \end{itemize}
        In the case~\ref{ex:13321b}, by Proposition~\ref{ref:1m321:prop} we see
        that the $\Gmult$-fixed locus near this point has dimension at most $8$,
        while the degree zero part of the tangent space at $[I]$ has dimension
        $10$. By Proposition~\ref{ref:tangentSpacesToStrata:prop} the strictly
        positive part of this tangent space also has dimension $10$; in
        particular, the fiber has dimension at most $10$. By the
        base-and-tangent-to-fiber estimate, the stratum
        of $[I]$ has dimension at most $8+10 = 2\cdot (10-1)$. Suppose
        that the stratum has dimension exactly $2\cdot (10-1)$.
        By the estimation above, the $\Gmult$-fixed locus near $[I]$ has
        dimension exactly $8$. By the bounds from
        Proposition~\ref{ref:1m321:prop} this means that $[I]\in
        \Hilb^{\Gmult}_{10}(\mathbb{A}^3, 0)$ lies in the closure
        $\mathcal{V}$ of the irreducible locus
        from Proposition~\ref{ref:1m321:prop}\ref{it:1m321easy2} and that there
        exists an open neighbourhood of $[I]\in \mathcal{V}$ such that the
        the fiber of $\pi$ is smooth $10$-dimensional over any point of this
        neighborhood.
        The
        \emph{strictly positive} degree parts of the tangent spaces are
        constant for this open neighbourhood.
        Consider the ideal $I'$ apolar to $y_1^2 + y_1y_3, y_2y_3^3 + y_3^4$.
        It lies in $\mathcal{V}$ and its lex-initial ideal is $I$, so the
        $T$-orbit of $[I']$ is a curve in $\mathcal{V}$ which intersects the
        above neighbourhood nontrivially.
        Yet, the
        tangent spaces to the stratum at $[I']$ and $[I]$ have dimensions differing in degree $1$, which gives a
        contradiction. The very same strategy works in the
        case~\ref{ex:133211b} with Remark~\ref{ref:1m3211:rem}, the ideal $I'$ apolar to $y_1^2 + y_1y_3, y_2y_3^2,
        y_3^5$ and degree $2$.

        In the case~\ref{ex:13421} the degree zero part of the tangent space
        is $12$-dimensional, while the
        possible homogeneous inverse
        systems are at most $10$-dimensional by
        Proposition~\ref{ref:13421:prop}, so by the base-and-tangent-to-fiber
        estimate the stratum has dimension at
        most $19$. This concludes the last case and the whole proof.
    \end{proof}

\begin{proof}[Proof of Theorem~\ref{ref:irreducibility}]
   By~\cite{DJNT} we have $\Hilb_k(\mathbb{A}^3) = \HilbSmk(\mathbb{A}^3)$. 
    By
    Proposition~\ref{ref:componentsLowerBound:prop} each of the
    irreducible components of $\HilbSmk(\mathbb{A}^3, 0)$ has dimension at least $2(k-1)$.
    Take any such component $\mathcal{Z}$. There is a stratum $\HilbFunc{H}{3}$ which
    contains an open subset of this component, so the dimension
    of $\mathcal{Z}\cap \HilbFunc{H}{3}$ is at least $2(k-1)$. Consider a
    $\Gmult$-fixed point $[I]$ of this intersection.
    The ideal $I$ is
    graded and
    the nonnegative part of its tangent space to the Hilbert scheme has dimension at
    least $2(k-1)$.
    The generic initial ideal $I_0$ of $I$ is
    Borel-fixed~\cite[\S15.9]{EisView}. By semicontinuity, both
    its stratum and the
    nonnegative part of its tangent space have dimension at least $2(k-1)$.
    But by
    Propositions~\ref{ref:tangentSpacesToStrata:prop},~\ref{ref:smallExceptionalStrata:prop}
    there is a single such point, corresponding to $I_0 = (x_1, x_2,
    x_3^k)$. It follows that the generic initial
    ideal of $I$ is in the curvilinear locus, hence $I$ itself is in this
    locus. Thus the stratum of $I$ coincides with the stratum of $I_0$ which is
    equal to the curvilinear locus, so the chosen component lies in this locus.
    This concludes the proof.
\end{proof}

\section{Construction of regular maps to Grassmannian}\label{sec:regularMaps}

In this section, we derive some existence statements for regular maps.
They generalize the upper bounds from~\cite{bjjm} and are proven using a
similar method, although the key Proposition~\ref{ref:main:prop}, which
generalizes~\cite[Lemma~2.3]{bubu2010} (that actually refers
to~\cite{bgl2010}), has to be proven in a bit subtler
way.

\subsection{Reduction to small socle}\label{sec:reduce}

\newcommand{\spannpar}[2]{\spann{#1}_{#2}}%
    Let $V$ be a $\KK$-vector space and $\Gr(\tau{}, V)$ be the Grassmannian of its
    $\tau{}$-dimensional subspaces. Let $\mathcal{U}$ be the universal subbundle, which is a rank $\tau{}$ vector bundle over $\Gr(\tau{},
    V)$. For a zero-dimensional scheme $Z$ with a morphism $f\colon Z \to \Gr(\tau{}, V)$, let
    $\mathcal{U}_Z$ be the pullback of
    this bundle to $Z$ so that we have
    \[
        \begin{tikzcd}
            \mathcal{U}_Z\arrow[r, hook, "closed"] & Z \times V\arrow[r, "\pr_2"] & V
        \end{tikzcd}
    \]
    and let $\spannpar{Z}{f}$ be the linear space spanned in $V$
    by $\pr_2(\mathcal{U}_Z)$, where $\pr_2$ is the second projection.
    We stress that $\spannpar{Z}{f}$ depends not only on $Z$ but also on the
    morphism $f\colon Z\to \Gr(\tau{}, V)$.

    \begin{example}\label{ex:reducedScheme}
        The motivation for the symbol $\spannpar{Z}{f}$ is the case when $Z$ is a tuple
        of $k$ points. Then the map $f\colon Z\to \Gr(\tau, V)$ corresponds to choosing
        $k$ subspaces $W_1, \ldots ,W_k \subset V$ and so $\spannpar{Z}{f} = W_1 +
         \ldots + W_k$. In particular, for $\tau = 1$ and for $Z \into \Gr(1,
         V) = \mathbb{P}(V)$ we get the usual notion of projective span of
         $Z$.
    \end{example}

    For a zero-dimensional irreducible scheme $Z$, let $\tau(Z) := \tau(H^0(Z, \OO_Z)) =
    \dim_{\KK} \Soc(H^0(Z, \OO_Z))$. For example, $\tau(Z) = 1$ if and only if $Z$ is
    Gorenstein.
    The following is the key result that will allow us to reduce from
    $\Hilb_k(\mathbb{A}^n, 0)$ to its sublocus $\Hilb_k^{\tau}(\mathbb{A}^n,
    0)$. It is a generalization
    of~\cite[Lemma~3.5 (iii)$\implies$(i)]{bgl2010}, which dealt with the case
    $\tau=1$.
    \begin{proposition}\label{ref:main:prop}
        Let $Z$ be a zero-dimensional irreducible scheme with $\tau(Z) > \tau{}$
        together with a morphism $f\colon Z\longrightarrow \Gr(\tau, V)$. Then
        $\spannpar{Z}{f} =
        \bigcup_{Z' \subsetneq Z} \spannpar{Z'}{f|_{Z'}}$. Consequently, we
        have
        \[
            \spannpar{Z}{f} = \bigcup\left\{ \spannpar{Z'}{f|_{Z'}} \ |\ Z'
        \subset Z, \tau(Z')\leq\tau \right\}.
        \]
    \end{proposition}

    For the proof, we need the following auxiliary lemma.
    \begin{lemma}\label{ref:functionals:lem}
        Let $Z$ be a zero-dimensional irreducible scheme with $\tau(Z) > \tau{}$.
        Let $\EE$ be a rank $\tau{}$ locally free sheaf on $Z$ and
        let $E = H^0(Z, \EE)$. Finally, let $\varphi\colon E\to \KK$ be a $\KK$-linear map.
        Then there exists a proper subscheme $Z' \subsetneq Z$ and a
        $\KK$-linear map $H^0(Z', \EE|_{Z'}) \to \KK$ such that $\varphi$
        factors as
        \[
            H^0(Z, \EE)\onto H^0(Z', \EE|_{Z'}) \to \KK,
        \]
        where the first map comes from restricting $\EE$ to $Z'$.
    \end{lemma}
    \begin{proof}
        The sheaf $\EE$ is actually free, so we assume $\EE = \OO_Z^{\oplus
        \tau{}}$ and consequently $E = H^0(Z, \OO_Z)^{\oplus \tau{}}$. Let $i_p\colon
        H^0(Z, \OO_Z)\to E$ be the injection of the $p$-th factor. The space
        \[
            \bigcap_{p=1}^\tau{} \ker(\varphi \circ i_p) \subseteq H^0(Z,
            \OO_Z)
        \]
        is a corank $\leq \tau{}$ subspace. By assumption,
        $\dim_{\KK}\Soc(H^0(Z, \OO_Z)) > \tau{}$, so this subspace intersects
        the socle nontrivially. Let $r$ be a non-zero
        element of this intersection. The subscheme $Z' = V(r)$ is the
        required one.
    \end{proof}

    \begin{proof}[Proof of Proposition~\ref{ref:main:prop}]
        \def\II{\mathcal{I}}%
        Let $0\to \mathcal{U}\to \OO_{\Gr(\tau{}, V)}\otimes V \to Q\to 0$ be the universal
        sequence of locally free sheaves on $\Gr(\tau{}, V)$, where $\mathcal{U}$ has rank
        $\tau{}$. Dualizing and taking symmetric powers, we obtain the sequence
        \[
            0\to \mathcal{K}\to \OO_{\Gr(\tau{}, V)} \otimes \Sym(V^*)\to
            \Sym(\mathcal{U}^*)\to 0
        \]
        where $\mathcal{K} = (Q^*)$ is the ideal sheaf generated by degree one
        elements and with $\mathcal{K}_1 = Q^*$; it is exactly the ideal sheaf
        of $\mathcal{U} \into \Gr(\tau{},
        V)\times V$. Restricting this sequence to $Z$ via $f$ we obtain the sequence
        \[
            0\to \mathcal{K}_Z\to \OO_{Z} \otimes \Sym(V^*)\to
            \Sym(\EE)\to 0
        \]
        where $\EE = \mathcal{U}^*|_{Z}$ and so $\mathcal{K}_Z$ is the ideal sheaf of
        $\mathcal{U}_Z \into Z \times V$.
        The ideal of $\pr_2(\mathcal{U}_{Z})$ is $\mathcal{K}_Z \cap \Sym(V^*)$, where
        $\Sym(V^*)\into \OO_Z \otimes_{\KK} \Sym(V^*)$ is the usual inclusion.
        Note that this ideal is homogeneous.
        Finally, the linear span of $\pr_2(\mathcal{U}_{Z})$ is given by taking only the
        linear part of this ideal. So we discard the nonlinear parts of the
        above sequence and take global sections to obtain
        \[
            \begin{tikzcd}
                && V^* \arrow[d]\arrow[rd, "\varphi"]\\
                0\ar[r] & H^0(\mathcal{K}_Z)_1 \ar[r] & H^0(Z, \OO_Z) \otimes_{\KK} V^*\arrow[r] & H^0(Z,
                \EE)\ar[r] & 0
            \end{tikzcd}
        \]
        A $\KK$-point $x\in V$ lies in $\spann{\pr_2(\mathcal{U}_{Z})}$ if and only if
        $\mm_{x}$ contains $\ker \varphi$. For such a point we complete the
        above diagram to
        \[
            \begin{tikzcd}
                && V^* \arrow[d]\arrow[r, "ev_{x}"] & \KK\\
                0\ar[r] & H^0(\mathcal{K}_Z)_1 \ar[r] & H^0(Z, \OO_Z) \otimes_{\KK} V^*\arrow[r] & H^0(Z, \EE)\arrow[u,
                "\varphi_x"]\ar[r] & 0
            \end{tikzcd}
        \]
        where $\varphi_x$ is a map making the diagram commutative.
        By Lemma~\ref{ref:functionals:lem}, there is a proper closed subscheme $Z'
        \subset Z$ such that the relevant part of the above diagram completes to a commutative
        diagram
        \[
            \begin{tikzcd}
                V^* \arrow[d]\arrow[r, "ev_{x}"] & \KK\\
                H^0(Z', \OO_{Z'}) \otimes_{\KK} V^*\arrow[r, two heads] &
                H^0(Z', \EE|_{Z'})\arrow[u,
                "linear"]\\
                H^0(Z, \OO_Z) \otimes_{\KK} V^*\arrow[r, two heads]\arrow[u,
                two heads] & H^0(Z,
                \EE)\arrow[uu,
                "\varphi_x"', bend right=80]\arrow[u, two heads]
            \end{tikzcd}
        \]
        Reversing the above argument, we get that $x$ lies in the linear span
        of $\pr_2(U_{Z'})$ as claimed.
    \end{proof}

    \begin{example}
        \def\FF{\mathcal{F}}%
        Let $Z = \Spec(\KK[x_1,x_2,x_3]/(x_1, x_2, x_3)^2)$, so that $\tau(Z) =
        3$. Let $V$ be a $\KK$-vector space.
        A morphism $f\colon Z\to \Gr(2, V)$ corresponds to a rank two
        locally free subsheaf $\FF\subseteq \OO_Z \otimes_{\KK} V$ such that
        $(\OO_Z\otimes_{\KK} V)/\FF$ is locally free. Since $Z$ is
        zero-dimensional, the sheaf $\FF$ is free and determined by its space
        of global
        sections $H^0(\FF)$ which is spanned as an $H^0(\OO_Z)$-module by two
        elements
        \[
            e_{11} + e_{12}x_1 + e_{13} x_2 + e_{14} x_3,\ e_{21} + e_{22}x_1 + e_{23}x_2 +
            e_{24}x_3 \in \KK[x_1,x_2,x_3]/(x_1, x_2,
            x_3)^2\otimes_{\KK} V,
        \]
        where $e_{11}, \ldots ,e_{14}, e_{21}, \ldots ,e_{24}\in V$.
        By definition of $f$, we have $\mathcal{U}_Z = \FF$ as subsheaves of
        $\OO_Z \otimes V$,
        so that $\EE$ from the proof is equal to $\FF^*$ and the map
        $\varphi\colon V^*\to H^0(\EE)  \simeq  H^0(\FF)^*$ is given by
        \[
            \varphi(v^*) = (v^*(e_{11}), \ldots ,v^*(e_{14}), v^*(e_{21}), \ldots ,v^*(e_{24})).
        \]
        It follows that $\spannpar{Z}{f} = \spann{e_{11}, \ldots ,
        ,e_{24}}$. Consider a point $p = \sum_{1\leq i\leq 2,1\leq j\leq 4}
        \alpha_{ij} e_{ij}$
        in this span, where $\alpha_{ij}\in \KK$.
        Let $(\lambda_1, \lambda_2, \lambda_3)\in \KK^3$ be a nonzero vector
        perpendicular to $(\alpha_{12}, \alpha_{13}, \alpha_{14})$ and
        $(\alpha_{22}, \alpha_{23}, \alpha_{24})$. Reversing the above argument we get that $p$ lies in $\spannpar{Z'}{f|_{Z'}}$ where
        $Z' = V(\lambda_1 x_1 + \lambda_2 x_2 + \lambda_3
        x_3)\subset Z$.
    \end{example}

\subsection{Existence of regular maps and its connection with the Hilbert scheme}
In this section $\KK = \mathbb{C}$ and we consider $\mathbb{C}^n$ mostly as a
topological space with Euclidean topology. We also consider maps $X\to
\Gr(\tau, V)$ for more general $X$. In fact, we will mostly use the case $X =
\mathbb{P}^n$, see Corollary~\ref{ref:existence:cor}, but it seems more
natural to give the general definitions. Correspondingly, we will
consider $\Hilb_k(\mathbb{P}^n, p)$ where $p$ is a fixed closed point, and similar loci instead of
$\Hilb_k(\mathbb{A}^n, 0)$.
 \begin{definition}\label{def:algebraicKregular}
     Let $X$ be a topological space.
        A map $f\colon X\to \Gr(\tau, V)$ is
        \emph{$k$-regular} if
        for every tuple of $k$ distinct points $x_1, \ldots ,x_k\in
        X$ the linear space $\spann{f(x_1), \ldots ,f(x_k)}$ is
        $\tau k$-dimensional. Let $X$ be an algebraic variety. A morphism $f\colon X\to \Gr(\tau,
        V)$ is \emph{strongly $k$-regular} if for every zero-dimensional smoothable
        subscheme $Z\subset
        X$ of degree $k$ the linear space $\spannpar{Z}{f|_{Z}}$ is $\tau
        k$-dimensional.
    \end{definition}
    In light of Example~\ref{ex:reducedScheme}, strongly regular maps are
    regular. Moreover, the existence of a regular map for a certain $N$ implies the existence of this map for any integer $\geq N$.
    The definition of a $k$-regular map makes sense for every map of sets
    $f\colon X(\mathbb{C})\to \Gr(\tau, V)(\mathbb{C})$.

    Let $f\colon X\to \mathbb{P}^{N-1} = \mathbb{P}(\mathbb{C}^N)$ be a
    a morphism of varieties. Let
    $\mathbb{C}^{\tau N} = \bigoplus_{i=1}^{\tau}\mathbb{C}^N e_i$ and
    define the morphism $f^{\tau}\colon X\to \Gr(\tau,
    \mathbb{C}^{\tau N})$ by
    \[
        f^{\tau}(x) = \spann{\widehat{f(x)}e_1, \widehat{f(x)}e_2, \ldots ,
            \widehat{f(x)}e_\tau},
    \]
    where $\widehat{f(x)}\in \mathbb{C}^N$ is any element of the line $f(x)\in
    \mathbb{P}(\mathbb{C}^N)$. More formally, the morphism $f^{\tau}$ is defined as
    follows. Over $\mathbb{P}^{N-1}$ we have the universal line subbundle
    $\mathcal{U}\into \mathbb{P}^{N-1} \times \mathbb{C}^{N}$
    whose fiber over a point of $\mathbb{P}^{N-1}$ is the corresponding line
    in $\mathbb{C}^N$. In fact, this is a special case of the universal
    subbundle for a Grassmannian, as in the previous section. Pulling back
    $\mathcal{U}$
    via $f$ we get a subbundle $\mathcal{U}_f\into X \times
    \mathbb{C}^N$. The direct sum $\mathcal{U}_f e_1 \oplus  \ldots \oplus
    \mathcal{U}_f
    e_{\tau}\into \bigoplus_{i=1}^{\tau} \mathbb{C}^Ne_i$ is a bundle over
    $X$ that is a rank $\tau$ subbundle of $\mathbb{C}^{\tau
    N}$ so it induces a morphism
    $X\to \Gr(\tau, \mathbb{C}^{\tau N})$.

    When we view $\mathbb{P}^{N-1}$ as $\Gr(1, \mathbb{C}^N)$,
    Definition~\ref{def:algebraicKregular} gives us notions of $k$-regularity
    for $f$. Explicitly, $f$ is $k$-regular if for every $k$ distinct points of
    $X$, their images under $f$ span a $(k-1)$-dimensional
    projective subspace and $f$ is strongly $k$-regular if
    for every smoothable zero-dimensional subscheme $Z\subset X$ of degree $k$, its
    image $f(Z)$ spans a $(k-1)$-dimensional projective subspace.

    \begin{lemma}\label{ref:kregularTotau:lem}
        If $f$ is $k$-regular then $f^{\tau}$
        is $k$-regular.
        If $f$ is strongly $k$-regular then $f^{\tau}$ is strongly
        $k$-regular.
    \end{lemma}
    \begin{proof}
We first prove the second statement. Take a zero-dimensional smoothable subscheme $Z\subset X$ of degree
$k$. By strong $k$-regularity, $W := \spannpar{Z}{f}\subset \mathbb{C}^{N}$ is a
        $k$-dimensional linear
        space. The vector space $\spannpar{Z}{f^{\tau}}$ contains the
        subspaces $W e_i$ for
        all $i$, hence it contains the $\tau k$-dimensional vector space
        $\bigoplus_{i=1}^{\tau} W e_i$, which proves that $f^{\tau}$ is
        strongly $k$-regular. For the proof of $k$-regularity, one may restrict to
        $Z$ being a tuple of points and repeat the argument (one may also do a much
        more elementary argument).
    \end{proof}
    \begin{corollary}\label{ref:existence:cor}
        For every $\tau, n, k$ there exists an $N$ and a strongly $k$-regular
        morphism
        $\mathbb{P}^{n}\to \Gr(\tau, \mathbb{C}^N)$.
    \end{corollary}
    \begin{proof}
        Follows from Lemma~\ref{ref:kregularTotau:lem} and the existence of
        strongly $k$-regular morphisms from $\mathbb{P}^{n}$ to
        $\mathbb{P}^{N-1}$, see~\cite[Lemma~5.10]{bjjm}.
    \end{proof}

    Fix a closed point $p\in \mathbb{P}^n$, and consider a strongly $k$-regular
    morphism $F\colon \mathbb{P}^n\to \Gr(\tau, \mathbb{C}^N)$. It induces a map
    \[
        F_p\colon \HilbSmk(\mathbb{P}^n, p)\to \Gr(k\tau, \mathbb{C}^N)
    \]
    that sends $[Z]$ to $\spannpar{Z}{F}$. We pull back the
    universal subbundle of the Grassmannian via $F_p$ to obtain a bundle
    $\mathcal{U}_p$ on $\HilbSmk(\mathbb{P}^n, p)$ and projectivise this
    bundle to obtain
    \[
        \begin{tikzcd}
            \mathbb{P}(\mathcal{U}_{p})\arrow[r, hook] & \HilbSmk(\mathbb{P}^n, p) \times
            \mathbb{P}^{N-1}\arrow[d]\\
            & \HilbSmk(\mathbb{P}^n, p)
        \end{tikzcd}
    \]
    The \emph{areole} is the image of $\mathbb{P}(\mathcal{U}_p)$ in
    $\mathbb{P}^{N-1}$. As a set, it is the union of
    $\PP(\spannpar{Z}{F|_{Z}})$
    for $Z$ ranging over all zero-dimensional smoothable degree $k$ subschemes
    $Z\into\mathbb{P}^n$ supported only at the point $p$. We denote the
    areole by $\mathfrak{a}_{k,p} = \mathfrak{a}_{k,p}(F)$. By construction, we
    have $\dim \mathfrak{a}_{k,p} \leq \dim \mathbb{P}(\mathcal{U}_p) = \tau k-1+\dim
    \HilbSmk(\mathbb{P}^n, p)$. 
 The following gives a key improvement of the upper bound on the
    dimension of $\mathfrak{a}_{k,p}$.
    \begin{proposition}\label{ref:bound:prop}
      With the above notations, we have 
      $$\dim \mathfrak{a}_{k,p} \leq \max \lbrace \tau i-1+\dim
        \Hilb^{\tau}_i(\PP^n,p)\ | \ 1\leq i\leq k\rbrace.$$
    \end{proposition}
      \begin{proof}
 Using Proposition~\ref{ref:main:prop}, we have
 \[
     \mathfrak{a}_{k,p}= \overline{\bigcup\left\{\PP(\spannpar{Z}{F|_Z})\ |\ [Z]\in \HilbSmk(\PP^n,p)\right\}}
        \subseteq\bigcup_{1\leq i\leq k}\
        \overline{\bigcup\left\{\PP(\spannpar{Z'}{F_{Z'}})\ |\ Z'\in\Hilb_{i}^{\tau}(\PP^n,p)\right\}}
\]
 and so
$\dim \mathfrak{a}_{k,p} \leq \max \lbrace \tau i-1+\dim
\Hilb^{\tau}_i(\PP^n,p)\ | \ 1\leq i\leq k\rbrace$. 
    \end{proof} 
    \begin{proposition}[Reduction of $N$]\label{ref:reductionOfN:prop}
        There exists a
        $k$-regular continuous map $\mathbb{C}^n\to \Gr(\tau,
        \mathbb{C}^{M+1})$ for $M =
        \dim \mathfrak{a}_{k,p}$.
    \end{proposition}
    It should be stressed that the obtained $k$-regular map is usually
    not algebraic.
    \begin{proof}
        This is a straightforward generalization of~\cite[Theorem~5.7]{bjjm};
        we sketch the argument below.
        A general subspace $W \subset \mathbb{C}^{N}$ of dimension
        $N-M-1$ satisfies $\mathbb{P}(W) \cap \mathfrak{a}_{k,p} = \emptyset$,
        since the dimensions of these two varieties sum up to $N-M-2+M = N - 2
        < N-1$. Consider the projection $\pi\colon \mathbb{C}^{N}\onto
        \mathbb{C}^N/W$; the latter space is isomorphic to
        $\mathbb{C}^{M+1}$. This projection induces a rational map $\pi\colon \Gr(\tau,
        \mathbb{C}^N)\dashrightarrow \Gr(\tau, \mathbb{C}^N/W)$. Since
        $\mathbb{P}(W)$ is disjoint from the areole, it is, in particular,
        disjoint from the subspace corresponding to $F(p)$, so the composed map
        $F_W = \pi\circ F\colon \mathbb{P}^n\dashrightarrow \Gr(\tau, \mathbb{C}^{N}/W)$ is
        a well-defined morphism on a Zariski-open neighborhood of $p$. Suppose that for every Euclidean-open ball
        $B_{\varepsilon}$ of radius $\varepsilon$
        around $p$ the map $F_W|_{B_\varepsilon}\colon B_{\varepsilon}\to \Gr(\tau,
        \mathbb{C}^{N}/W)$ is not $k$-regular.
        For every $\varepsilon$ choose a $k$-tuple of points $x_1(\varepsilon), \ldots
        ,x_k(\varepsilon)\in B_{\varepsilon}$ such that $x_1(\varepsilon),
        \ldots ,x_k(\varepsilon)$ witnesses the non-regularity, so the
        corresponding subspaces do not span a $k\tau$-dimensional space in
        $\mathbb{C}^N/W$.
        Taking a countable sequence $(\varepsilon_n)$ converging to zero, we
        obtain a map $\mathbb{N}\to \HilbSmk(\mathbb{P}^n)$ that sends
        $\varepsilon_n$ to $\{x_1(\varepsilon_n), \ldots ,
        x_k(\varepsilon_n)\}$. Since the points $x_i(\varepsilon_n)$ converge
        to $p$ with $n\to \infty$, the Zariski closure of the image of
        $\mathbb{N}\to \HilbSmk(\mathbb{P}^n)$
        contains a point $[Z]$ corresponding to a subscheme $Z\subset
        \mathbb{P}^n$ supported only at $p$. By semicontinuity, the span
        $\spannpar{Z}{F_W} \subset \mathbb{C}^N/W$ is also of dimension less than
        $k\tau$, hence $\spannpar{Z}{F}$ intersects $W$ in a nonzero
        subspace, which is a
        contradiction with the choice of $W$.
        The obtained contradiction shows that for some $\varepsilon$ the
        continuous map
        $F_W|_{B_{\varepsilon}}$ is $k$-regular. Composing with a
        homeomorphism $B_{\varepsilon} \simeq \mathbb{C}^{n}$ we obtain the
        required map.
    \end{proof}
    Putting together
Propositions~\ref{ref:reductionOfN:prop} and \ref{ref:bound:prop} we obtain the
following statement.
  \begin{corollary}\label{ref:bound:cor}
        For every $\tau, n, k$ there exists a $k$-regular map $\mathbb{C}^n\to
        \Gr(\tau, \mathbb{C}^{N+1})$ for $N = \max \lbrace \tau i-1+\dim
        \Hilb^{\tau}_i(\PP^n,p)\ | \ 1\leq i\leq k\rbrace$.
    \end{corollary}
 
\section{The scheme $\Hilb_k^2(\AA^n,0)$ has expected dimension for $k\leq 11$}
In this section, we prove the upper bound $(n-1)(k-1)$ on the dimension
of  $\Hilb^{2}_k(\AA^n,0)$ for $k\leq 11$. Subdividing the
$\Hilb_k(\AA^n,0)$ according to the Hilbert function, as described in~\eqref{eq:mainconcrete},
it is enough to prove this bound on the dimension of
\[
    \Hilb_H^{2}(\mathbb{A}^n, 0) := \HilbFunc{H}{n}\cap \Hilb_k^{2}(\AA^n,0)
\]
for every possible Hilbert function with $k=\sum H\leq
11$ and $H(s)\leq 2$. Such Hilbert functions are subdivided into three groups as in the following subsections.

\subsection{Hilbert functions with $H(3) = 1$}

First, we consider the Hilbert functions with $H(3) = 1$. Such a
function has the form $(1, H(1), H(2), 1, 1,  \ldots , 1)$ and so any minimal dual
generating set has, without loss of generality, one ``large degree'' element
$f$ and other elements of degree at most two; informally speaking such an
algebra is close to the Gorenstein algebra given by the dual generator $f$.
The Hilbert function of the apolar algebra of $f$ is $(1, a, b, 1, 1,  \ldots
, 1)$ and it follows from the general theory of Iarrobino's symmetric
decomposition (e.g.~\cite[Remark~(2), p.~1532]{cjn}) that $a\geq b$. We will
also assume $a = n$, as this is the only case of interest for us.

\begin{proposition}\label{ref:locusOfGorensteins:prop}
    Let $\HilbFuncGor{H}{n}$ be the locus of Gorenstein algebras with the Hilbert
    function $H = (1,
    n, b, 1, 1, \ldots,1)$. Let $s$ be the largest number such that
    $H(s)\neq 0$.
    Then the dimension of $\HilbFuncGor{H}{n}$ is
    \[
        (n-1)(s-3)+(n-b)b+\binom{b+2}{3}-1+\binom{n+2}{2} - (1+n+b).
    \]
\end{proposition}
\begin{proof}
    \def\mmr{\mm_R}%
    \def\StandardForms{\mathrm{StrdForms}}%
    Every such Gorenstein algebra is given by an ideal $I = \Ann(f)$ for a
    single polynomial $f\in S$.
    By~\cite[Thm 5.3AB]{iarrobino_associated_graded} (see also \cite[Proposition~3.3]{cjn}) there exists a \emph{standard form of} $f$. This is a form
    $g\in S$ of the shape
    \begin{equation}\label{eq:standardForms}
        g = y_1^{s} + a_{s-1}y_1^{s-1} +  \ldots + a_{4}y_1^{4} + c + q,
    \end{equation}
    where $a_i\in \KK$, $c\in \KK[y_1, \ldots ,y_b]_3$ and $q\in \KK[y_1, \ldots ,y_n]_{\leq
    2} = S_{\leq 2}$ such that there exists an automorphism $\varphi\colon
    \hatR\to \hatR$ satisfying $\varphi^{\vee}(g) = f$,
    see~\cite[\S2.2]{Jel_classifying} for the details about the transformation $\varphi^{\vee}\colon S\to S$ dual to $\varphi$. This last
    condition can be rephrased as $I = \varphi(\Ann(g))$. Conversely, for a $g$ as
    in~\eqref{eq:standardForms} with $c$, $q$ general and for any $\varphi$, we get a form $\varphi^{\vee}(g)$ which is a dual
    generator of a Gorenstein algebra with the correct Hilbert function.
    It remains to count the possible $g$ and
    $\varphi$.
    By the explicit form of $\varphi^{\vee}$ as in~\cite[\S2.2, Equation~(5)]{Jel_classifying} it is enough to
    consider $\varphi\in \Aut(\hatR/\mm^{s+1})$.
    Let $\StandardForms \subset S_{\leq s}$ be the affine space of $g$ in the
    form~\eqref{eq:standardForms}.
    The parameter space for
    $f$ is the image of the map
    \[
        P\colon \Aut(\hatR/\mm^{s+1}) \times \StandardForms \to S_{\leq s}
    \]
    that sends $(\varphi, g)$ to $\varphi^{\vee}(g)$.
    Take a general $(\varphi, g)$ in the domain of $P$.
    Since $\KK$ has
    characteristic zero, the tangent map $dP$ is surjective onto
    $T_{\varphi^{\vee}(g)} \im P$ and its image has dimension equal to $\dim \im P$.
    Since $P$ is equivariant with respect to $\Aut(\hatR/\mm^{s+1})$, we can
    and do assume
    $\varphi = \id$.
    By~\cite[Proposition~2.18]{Jel_classifying} the image of $dP$ is given by
    \begin{align*}
        &\sum_{i=1}^n y_i (\mm\circ g) + \spann{y_1^s,  \ldots , y_1^4} + \KK[y_1, \ldots ,y_b]_3 + \KK[y_1,
        \ldots ,y_n]_{\leq 2} \\&= \spann{y_1^i\ell\ |\ \ell\in S_1, 3\leq i\leq
        s-1} + S_1y_1^2 + \sum_{i=1}^n y_i\cdot \spann{x_j\circ g\ |\ j=2,3, \ldots ,n} + \KK[y_1, \ldots ,y_b]_3 + \KK[y_1,\ldots ,y_n]_{\leq 2}.
    \end{align*}
    This space has dimension $n(s-3)+(n-b)+(n-b)(b-1)+\binom{b+2}{3}+\binom{n+2}{2}$.
    Using Proposition~\ref{ref:polynomialsAndGorensteinAlgebras:prop} we
    deduce that
    \begin{align*}
        \dim \HilbFuncGor{H}{n} &=
        n(s-3)+(n-b)b+\binom{b+2}{3}+\binom{n+2}{2} - (1+n+b+s-2)\\
        &=(n-1)(s-3)+(n-b)b+\binom{b+2}{3}-1+\binom{n+2}{2} - (1+n+b).\qedhere
    \end{align*}
\end{proof}

\begin{remark}
    Using the ideas from~\cite[\S2]{Jel_classifying}, it is not hard to
    actually describe the locus of possible $f$. For example, in the case
    $(1,3,2,1,1)$ we have $f = \ell_1^4 + \alpha\ell_1^2\ell_3 + F_3(\ell_1,
    \ell_2) + F_{\leq 2}(\ell_1, \ell_2, \ell_3)$, where $\alpha\in \KK$ and
    $F_i$ is homogeneous of degree $i$ and $\ell_i$ are linear forms. The above parameterization gives
    $3+1+(1+4)+10 = 19$ parameters and hence we obtain an $11$-dimensional
    family by Proposition~\ref{ref:polynomialsAndGorensteinAlgebras:prop}, in concordance with
    Proposition~\ref{ref:locusOfGorensteins:prop}. This more explicit argument
    works also in positive characteristics. We will not use it.
\end{remark}

\begin{corollary}\label{ref:H3eq1bound:cor}
    Let $H = (1, n,b, 1, 1,  \ldots, 1)$.
    The dimension of $\Hilb_H^{2}(\mathbb{A}^n, 0)$ is at most the maximum of
    \[
        \dim \HilbFuncGor{H-(0,0,1,0, \ldots )}{n} + \binom{n+1}{2} - b -
        1\qquad\mbox{and}\qquad \dim \HilbFuncGor{H}{n},
    \]
    where $H-(0,0,1,0, \ldots )$ means a Hilbert function that differs from $H$ only at
    position $2$ and is one less at that position.
\end{corollary}
\begin{proof}
    The inverse system of an element of $\Hilb_H^{2}(\mathbb{A}^n, 0)$ may have a minimal
    generator in degree two or not. This subdivides $\Hilb_H^{2}(\mathbb{A}^n, 0)$ into loci
    whose dimensions are as in the statement.
\end{proof}
\begin{corollary}\label{ref:H3eq1negligible:cor}
    For every Hilbert function $H = (1, a, b, 1, 1, \ldots, 1)$ with $\sum H
    \leq 12$, the bound from Corollary~\ref{ref:H3eq1bound:cor} is less than
    or equal to $(\sum H-1)(a-1)$, hence the corresponding locus is
    negligible.
\end{corollary}
\begin{proof}
    By Invariance of Codimension~\cite[Proposition~A.4]{bjjm} we may take $a =
    n$. Increasing the socle degree $s$ by one
    increases both the bound and the expected dimension by $n-1$, so it does
    not change negligibility, so we may take $s=3$.
    By an argument analogous to \cite[Remark~(2), p.~1532]{cjn} we have $b\leq
    a+1 = n+1$. Then the statement becomes a direct
    check of all possible cases using
    Proposition~\ref{ref:locusOfGorensteins:prop}.
\end{proof}
\begin{remark}
    The bound $12$ in Corollary~\ref{ref:H3eq1negligible:cor} is sharp: the
    inequality from this corollary is false already for
    $H = (1,5,6,1)$, as we will compute more explicitly in \S\ref{sec:counterexamples}.
\end{remark}

\subsection{Hilbert functions with $H(2)\leq 2$}

In this case we use the tangent space estimate.
    \begin{theorem}\label{ref:H2:thm}
        Let $H$ be a Hilbert function with $k = \sum H$.
        Suppose $H(2)\leq 2$. Then $\dim \HilbFunc{H}{n} < (k-1)(n-1)$.
    \end{theorem}
    \begin{proof}
        When $H(2) = 1$, then in fact $H(i) \leq 1$ for all $i\geq 2$,
        see~\cite[Remark~2.7]{cjn}, and the
        bound follows directly from the estimate in
        Proposition~\ref{ref:locusOfGorensteins:prop}.
        The case $H(2) = 2$ follows from
        Corollary~\ref{ref:H2eq2main:cor}.
    \end{proof}
    This theorem allows us to greatly reduce the number of cases that appear.
    It is also a partial strengthening of the Brian{\c{c}}on-Iarrobino result that
    implies a similar claim for $H(1)\leq 2$. Note that we have no assumptions
    on $k$ here, and the bound is general.

    To prove the theorem we apply several steps. Using invariance of
    codimension, we reduce to considering algebras with $H(1) =n$. We also
    assume $H(2) = 2$, then $H(i)\leq 2$ for all $i\geq 2$,
    see~\cite[Remark~2.7]{cjn}.
    The first step is to show that there are only a few isomorphism classes of
    \emph{graded} algebras with the required Hilbert function. The next step
    is to bound the strictly positive part of the tangent space at each such
    algebra, to control the fibers of the map $\pi$ sending an algebra to its
    associated graded.
    We will now analyze the tangent space to the Hilbert scheme
    at the points of our stratum. First, we find the Borel-fixed ones.

        \begin{lemma}\label{ref:canonicalFormForH2eq2:lem}
            Suppose $I\subset R = \KK[x_1, \ldots ,x_n]$ is a Borel-fixed
            ideal whose quotient algebra $A = R/I$ is zero-dimensional with $H(2) = 2$. Let $s$ be the socle degree of $A$ and
            let $t\leq s$  be the largest degree where $H(t) = 2$.
            Then the inverse system of $I$ is generated by the following
            elements: $S_1$, $y_n^{s}$, $y_{n-1}y_n^{t-1}$.
        \end{lemma}
        \begin{proof}
            The ideal $I$ is Borel-fixed, hence monomial. Since $H_{A}(2) =
            2$, the only possibility for $I_2$ is to contain all monomials
            except $x_n^2$ and $x_{n-1}x_n$. Then the
            inverse system $I^{\perp}$ satisfies $I^{\perp}_2 = \spann{y_{n}^2,
                y_{n-1}y_n}$. This
                shows that $I^{\perp}_{\geq 2} \subset \KK[y_{n-1}, y_n]$ and we
                reduce to the case $n=2$, so that $y_{n-1} = y_1$, $y_n
                =y_2$, and $I_2 = \KK x_{1}^2$.
            Since $x_1^2\in I$, we have $I^{\perp}_{\geq 2} \subset \spann{y_2^{a},
            y_1y_2^{b}\ |\ a,b\geq 0}$, so $I^{\perp}$ is generated by $S_1$
            and $y_2^a$ and $y_1y_2^b$ for some $a,b\geq 0$.
            Moreover, in all degrees where $I^{\perp}$ is non-zero, it is
            Borel-fixed hence contains the largest monomial, which is $y_2^a$.
            Therefore, we conclude that $a = s$. It follows from the Hilbert
            function that $b = t-1$.
        \end{proof}
        \begin{lemma}\label{ref:tangentForH2eq2:lem}
            Let $A = R/I$ be a graded algebra whose inverse system is given by
            $S_1$ together with $y_1^s$ and $y_1^{t-1}y_2$ for some $1\leq t\leq s$.
            Then the non-negative part of $T_A = \Hom_R(I,R/I)$ has Hilbert series
            \[
                T^{s-t-1} + \left( \binom{n+1}{2}-2 -(n-1)
                \right)(T^{t-2}+T^{s-2}) + (n-1)\sum_{i=2}^s H_A(i)T^{i-2},
            \]
            where $T^{\alpha}$ is interpreted as zero for $\alpha <
            0$.
        \end{lemma}
        \begin{proof}
            The ideal $I$ is generated by $x_2^2$, $x_1x_i$, $x_2x_i$,
            $x_ix_j$ for $i,j\geq 3$ and $x_1^tx_2, x_1^{s+1}$. For $2\leq
            a\leq t$
            the space $(R/I)_a$ has a basis $\{x_1^a, x_1^{a-1}x_2\}$, while
            for $t < a \leq s$ it has a basis $\{x_1^a\}$.
            Consider a homomorphism $\varphi\colon I\to R/I$ of degree $d=a-2\geq
            0$. The generator $x_1^{s+1}$ is sent to zero by degree reasons.
            The generator $x_1^{t}x_2$ is sent to some power of $x_1$, so
            contributes $1$ in degrees $0, 1,  \ldots , s-t-1$.

            Write $\varphi(x_ix_j) = \lambda_{ij}x_1^{a} +
            \mu_{ij}x_1^{a-1}x_2\mod I$ for any $x_ix_j\in I$ with $i \leq j$; here
            if $a > s$ or, respectively, $a > t$ we use the convention that
            $\lambda_{ij}$ or,
            respectively, $\mu_{ij}$, is zero.
            For $a = s$ the element $\lambda_{ij}$ can be chosen arbitrarily
            and for $a = t$ the element $\mu_{ij}$ can be chosen arbitrarily,
            as these two coefficients stand next to socle elements. Below
            we assume that those two specific coefficients are zero.

            If $i,j \geq 3$ then the syzygies $x_1(x_ix_j) = x_i(x_1x_j)$ and
            $x_2(x_ix_j) = x_i(x_2x_j)$ force $x_1\varphi(x_ix_j) = 0 =
            x_2\varphi(x_ix_j)$, so  $\lambda_{ij} = \mu_{ij} =0$.
            It remains to consider $i=1$ and $i=2$. Assume $j\geq 3$. The
            syzygy $x_1(x_2x_j) = x_2(x_1x_j)$ implies that
            \[
                (\lambda_{1j}-\mu_{2j})x_1^{a}x_2 + \lambda_{2j}x_1^{a+1}
                \equiv 0 \mod I.
            \]
            If $a-1 \geq t$, then $\mu_{2j} = 0$ by the convention above.
            Otherwise $x_1^ax_2\not\in I$ and so $\mu_{2j} = \lambda_{1j}$.
            If $a \geq s$, then $\lambda_{2j} = 0$ by convention, while for $a
            < s$ we get $\lambda_{2j} = 0$ because $x_{1}^{a+1}\not\in I$.
            It follows that $\varphi(x_2x_i)$ is uniquely determined by
            $\varphi(x_1x_i)$.

            The syzygy $(x_1^tx_2)x_2 = (x_1^t)x_2^2$ shows that the
            coefficient of $x_1^b$ in $\varphi(x_2^2)$ is uniquely determined
            by $\varphi(x_1^tx_2)$ for $b+t \leq
            s$ which gives an additional constraint for homomorphisms in
            degrees $0, 1,  \ldots,
            s-t-2$.
            Summing up, the homomorphism $\varphi$ is uniquely determined by
            choosing
            \begin{enumerate}
                \item two arbitrarily coefficients (next to socle
                    elements) for every quadric generator except $x_2^2$ and
                    $x_1x_i$ for $i\geq 3$. This contributes $\left(
                    \binom{n+1}{2}-2 - (n-1) \right)(T^{t-2}+T^{s-2})$,
                \item an arbitrary image of $x_1^tx_2$ which, taking into
                    account the
                    restrictions for $\varphi(x_2^2)$, contributes $T^{s-t-1}$,
                \item arbitrary images of $x_2^2$ and $x_1x_i$ for $i\geq 3$
                    which contributes $(n-1)\sum_{i=2}^s H_A(i)T^{i-2}$.
            \end{enumerate}
            Summing up the three contributions, we obtain the result.
        \end{proof}
        \begin{corollary}\label{ref:H2eq2main:cor}
            For $H(2) = 2$ the stratum $\HilbFunc{H}{n}$ has dimension
            at most
            \[
                (k-1)(n-1) - 1.
            \]
        \end{corollary}
        \begin{proof}
            By the tangent space estimate and semicontinuity, it is enough to
            bound the non-negative part of the tangent space at Borel-fixed points. By
            Lemma~\ref{ref:canonicalFormForH2eq2:lem} there is only one such
            point $[I]$
            and the bound on its tangent space is given
            in Lemma~\ref{ref:tangentForH2eq2:lem}:
            \begin{multline*}
                \dim_{\KK} T_{[I]}\HilbFunc{H}{n} \leq 1 +
                2\left( \binom{n+1}{2}-2-(n-1) \right) +(n-1)\sum_{i=2}^s
                H_A(i) =\\
                1 + (n+1)n-2(n+1)+(n-1)(k-1-n) = (n-1)(k-1) -
                1.\qedhere
            \end{multline*}
        \end{proof}
            Using Lemma~\ref{ref:canonicalFormForH2eq2:lem}, we could machine-check
            the tangent spaces for $k\leq 11$. However the argument in
            Corollary~\ref{ref:H2eq2main:cor} is quite clean and works for any $k$.

\subsection{The remaining Hilbert functions}\label{Bound}
In this subsection, we treat the special cases which are not covered
previously. Recall the map $\pi$ in~\eqref{eq:mainconcrete} that sends each local
algebra to its associated graded algebra. The following is an easy but
useful result.

\begin{proposition}[{\cite[Proposition~4.3]{ccvv}}]\label{ref:fibers:prop}
Let $H=(1,n,a,b)$ and $t={n+1\choose 2}-a$. Then every fiber of $\pi$ is irreducible of dimension $tb$.
\end{proposition}
\begin{proof}[Sketch of proof]
    There are $t$ quadric generators of any graded quotient
    $A$ of $\KK[x_1, \ldots ,x_n]$ with Hilbert function $H$ and they can be
    send to arbitrary elements of $A_3$ to obtain an element of the fiber: the
    conditions coming from syzygies are vacuous as $H(4) = 0$.
\end{proof}

\begin{proposition}\label{ref:1m32new:prop}
    Let $H=(1,n,3,2)$. The locus $\Hilb_H^{2}(\mathbb{A}^n, 0)$ is
    negligible.
\end{proposition}

\begin{proof}
    We are to prove that the dimension of the locus is at most $(n+5)(n-1)$.
    By Proposition~\ref{ref:fibers:prop}, the fibers
    of~$\pi$ have dimension $n(n+1) - 6$.
    For an algebra $R/I$ as in the statement, its inverse system is generated by two
    polynomials of degree three $f_1, f_2$. Therefore, the leading terms of all degree
    two elements of this system are partials of the leading terms of $f_1$,
    $f_2$. Therefore, the inverse system of $\pi(A)$ will have no minimal quadric generators.
    It is
    enough to prove that the locus of graded algebras
    with the Hilbert function $H$ and with no minimal quadric generators in the
    inverse system has dimension at most $(n+5)(n-1) - (n(n+1)-6) =
    3n+1$. To do this we will decompose it into locally closed subloci, according to the
    properties of the dual generators.

    Consider first the inverse systems which contain a cubic
            $F$ essentially in three variables. The Hilbert function of the
            apolar algebra of $F$ is $(1,3,3,1)$ and the inverse systems are
            parameterized by choosing a $3$-dimensional subspace $V$ of linear
            forms, choosing $F\in \mathbb{P}(\Sym^3 V)$, and choosing a codimension one subspace
            in the space of minimal cubic generators of $\Ann(F)$. For
            example, by Boij-So\"ederberg theory, it is known that there are at
            most two minimal cubic generators of $\Ann(F)$. Therefore, the
            above parameterization gives an upper bound of $\dim \Gr(3, n) + 9 + 1
            = 3n+1$. This concludes this case.

    Consider now the inverse systems where every cubic depends essentially on
    at most two variables.
    Here, consider first the inverse systems which
    contain a perfect cube $\ell^3$. Such an inverse system contains, as minimal
    generators, the cube $\ell^3$ and a cubic $F$ which depends essentially on
    two variables.
    This locus is parameterized by choosing a linear form $\ell$, a space $V$ of linear forms for $F$ and
    next choosing $F\in \mathbb{P}(\Sym^3V)$. This gives $\dim \Gr(1,n)+\dim \Gr(2,
    n) + 3 = 3n-2 < 3n+1$.

    Finally, consider the inverse systems where every cubic depends
    essentially on exactly two variables. Since $H(2) = 3 < 2\cdot 2$, the
    spaces of first-order partials for cubics must intersect. Therefore, we can parameterize
    the whole locus by choosing a $3$-dimensional subspace of linear forms,
    choosing two of its $2$-dimensional subspaces $V_1$, $V_2$, and choosing
    $F_i\in \mathbb{P}(\Sym^3V_i)$ for $i=1,2$. This gives a parameterization by
    $\dim \Gr(3, n) + 2 + 2 + 3 + 3 = 3n+1$ parameters.
\end{proof}
\begin{remark}
    The locus of algebras with Hilbert function $H = (1,n,3,2)$ without
    restrictions on $\tau$ is \emph{not}
    negligible: choosing two perfect cubes, an arbitrary quadric, and using
    Proposition~\ref{ref:fibers:prop} gives a locus of dimension
    $2(n-1)+\binom{n+1}{2} - 2 + n(n+1) -6$. Hence, keeping track of $\tau$ is
    essential.
\end{remark}

\begin{proposition}\label{ref:1442:prop}
Let $H=(1,4,4,2)$. The locus $\Hilb_H^{2}(\mathbb{A}^4, 0)$ is
negligible.
\end{proposition}
\begin{proof}
    In general terms, the proof is similar to the one of
    Proposition~\ref{ref:1m32new:prop}.
    We only consider the graded algebras and our aim is to show that their
    locus has dimension at most $10\cdot 3 - 6\cdot 2 = 18$.

    First, consider the inverse systems which contain a cubic $F$ essentially
    depending on $4$ variables, so that the corresponding ideal has Hilbert
    function $(1,4,4,1)$.
    By~\cite[Proposition~4.6]{casnati_notari_irreducibility_Gorenstein_degree_10}
    the ideal $\Ann(F)$ has at most three-dimensional space of cubic minimal
    generators. Moreover, if $\Ann(F)$ has \emph{any} cubic minimal
    generator, then $F$ is a limit of direct sums, by~\cite[Theorem 1.7]{bbkt}
    or by~\cite[Lemma~4.5]{casnati_notari_irreducibility_Gorenstein_degree_10}
    together with the form given in equation~\cite[(1)]{bbkt}. Such direct
    sums are parameterized by a locus of dimension $\dim\Gr(1, 4) + \dim
    \Gr(3, 4) + \dim(\mathbb{P}(\Sym^3\KK^1)\times
    \mathbb{P}(\Sym^3\KK^3))$ which gives in total $15$ parameters
    (choosing a decomposition $\KK^2 \oplus \KK^2$ yields
    a similar count). Together with the choice of a codimension one space of
    cubic minimal generators, this gives at most $17$ parameters. This
    concludes the subcase.
    From now on we assume that every cubic in the inverse system
    essentially depends on at most $3$ variables.

    Second, consider the inverse systems in which some cubic depends on at most two
    variables. They are parameterized in a naive manner by $\dim \Gr(3,4)
    + \dim \Gr(2,4) + \dim\mathbb{P}(\Sym^3\KK^3) +
    \dim\mathbb{P}(\Sym^3\KK^2) = 19$
    parameters. Moreover, a general such pair of cubics gives rise to an
    inverse
    system with Hilbert function $(1,4,5,2)$. Hence the cubics giving rise to
    $(1,4,4,2)$ are parameterized by a proper closed (perhaps reducible)
    subvariety which thus has dimension at most $18$. This concludes this
    case.

    Finally, consider the inverse systems in which every cubic depends
    essentially on exactly three variables.
    Consider such a system $F_1$, $F_2$ and assume $x_1\circ F_1 = 0$, $x_2
    \circ F_2 = 0$. Since $H(2) = 4$, the, a priori $6$-dimensional, space
    $x_2\circ F_1$, $x_3\circ F_1$, $x_4\circ F_1$, $x_1\circ F_2$, $x_3\circ
    F_2$, $x_4\circ F_2$ has dimension at most four. Therefore, there are two
    independent relations
    \[
        \ell_{1, 1} \circ F_1 = \ell_{2,1}\circ F_2 \qquad \mbox{ and }\qquad
        \ell_{1,2} \circ F_1 = \ell_{2, 2}\circ F_2.
    \]
    Without loss of generality, we may assume that $\ell_{i, *}$ does not
    contain $x_i$ for $i=1,2$.
    Note that $\ell_{1,*}$ are linearly independent, since otherwise some
    combination of $\ell_{2,*}$ annihilates $F_2$, which is impossible since
    $F_2$ depends essentially on three variables and $\ell_{2,*}$ do not
    contain $x_2$.
    We observe that for $i=1,2$ we have
    \[
        \ell_{1,i}\circ (x_2 \circ F_1) = x_2 \circ (\ell_{1,i}\circ F_1) =
        x_2\circ (\ell_{2,i}\circ F_2) = \ell_{2,i}\circ (x_2 \circ F_2) = 0.
    \]
    We deduce that the quadric $x_2\circ F_1\in \KK[y_2, y_3, y_4]$ is
    annihilated by a two-dimensional space $\ell_{1,*}\subset \KK[x_2, x_3,
    x_4]_1$ and thus $x_2 \circ F_1$ is \emph{a perfect square}. Therefore, our
    locus is parameterized by choosing two three-dimensional spaces of linear
    forms $V_1$, $V_2$ and choosing $F_i\in \mathbb{P}(\Sym^3V_i)$ such
    that $x_{3-i}\circ F_i$ is a perfect square, where $x_i = V_i^{\perp}$. The
    choices of $V_i$ and $F_i$ a priori give $2\cdot(3+9)$ parameters but
    \emph{$x_{3-i}\circ F_i$ is a perfect square} is a codimension three
    condition on $F_i$, so in fact we get $2\cdot (3+9-3) = 18$ parameters.
    This concludes the whole proof.
\end{proof}
\begin{remark}
    It seems surprisingly hard to properly describe the locus from
    Proposition~\ref{ref:1442:prop}. Devising a general tool to do this
    would be interesting.
\end{remark}

\begin{proposition}\label{ref:14321final:prop}
Let $H=(1,4,3,2,1)$. The locus $\dim \Hilb_{H}^2(\mathbb{A}^4,0)$
is negligible.
\end{proposition}

\begin{proof}
    The locus of homogeneous part of $\Hilb_{H}^2(\mathbb{A}^4,0)$ decomposes
    into four subloci as in Proposition~\ref{ref:1m321:prop}. We consider them
    case by case, but not in order. For each locus, we either bound the
    dimension of each fiber of $\pi$ directly or perform a degeneration argument to
    reduce to bounding the dimensions of special fibers and then compute the
    tangent space to the fiber.

    Case~\ref{ref:1m321:prop}\ref{it:1m321last}. Here the parameterization is straightforward. The
    associated graded system is generated by a quartic $F$ in two variables,
    say $F\in \KK[y_1, y_2]$, and linear forms. The inverse system of the fiber
    is generated by degree four polynomial $f = F + f_3 + f_2$ and linear
    forms. The part $f_2$ can be arbitrarily chosen modulo the partials of
    $F$, which gives a $7$-dimensional choice. The cubic part $f_3$ has to
    satisfy $x_3 \circ f_3$, $x_4\circ f_3$ being partials of $F$, hence $x_3
    \circ f_3, x_4\circ f_3\in \KK[y_1,y_2]$ so that $f_3\in \spann{y_1, y_2,
    y_3, y_4}\KK[y_1, y_2]_2$ lies in a $10$-dimensional space of choices.
    Together, we see that the fiber is at most $17$-dimensional, so the whole
    locus --- using the conclusion of Proposition~\ref{ref:1m321:prop} --- has dimension at most $17 + 8$.

    Case~\ref{ref:1m321:prop}\ref{it:1m321subtle}. Here, the associated graded system is generated
    by $\ell_1^4$ and $c\in \KK[\ell_1, \ell_2, \ell_3]$. Taking the limit at
    zero
    with respect to a $\Gmult$-action with weights of $\ell_2$, $\ell_3$
    equal and smaller than the weight of $\ell_1$, reduces our system to the
    case $c\in \KK[\ell_2, \ell_3]$. Then by classification of cubics in two
    variables, we have $c= \ell_2^3 + \ell_3^3$ or $c = \ell_2^2\ell_3$, up to
    coordinate change. We verify that in both cases the tangent to the fiber
    has dimension $18$, hence same holds for every fiber by semicontinuity,
    thus  --- using the result of Proposition~\ref{ref:1m321:prop} --- the
    whole locus has dimension at most $18+10$.

    By
    Proposition~\ref{ref:1m321:prop} there remain three possible types of
    such homogeneous inverse systems: $\ell_1^4 + \ell_2^4$, $q$ or $\ell_1^3\ell_2$, $q$ or $\ell_1^4$,
    $\ell_2^3$, $q$ and in each care the dimension of the homogeneous locus
    is at most $14$. For each inverse
    system, since we know that $q\not\in \KK[\ell_1, \ell_2]$, we may choose
    $\ell_3$ such that the monomial $\ell_3^2$ appears in $q$. Taking a limit
    of a $\Gmult$-action with weights of $\ell_1$, $\ell_2$, $\ell_4$ equal
    and greater
    than the weight of $\ell_3$, reduces our system to the case $Q =
    \ell_3^2$.
    For such a $q$, a direct tangent space check shows that
    for $I$ annihilating such an inverse system we have $\dim
    \Hom_R(I, R/I)_{>0}$ equal to $14, 14, 18$ respectively. Since $14 + 14 < 30$,
    we concentrate on the last case. Summing up, from now on we only consider
    inverse systems whose associated graded has the form $\ell_1^4$,
    $\ell_2^3$, $q$. By Proposition~\ref{ref:1m321:prop} the homogeneous such
    systems form a $13$-dimensional family. Again, since every such system degenerates to $\ell_1^4$,
    $\ell_2^3$, $\ell_3^2$, it is enough to prove the dimension estimate for
    that system.

    We subdivide this case into two. First, we fix a basis $y_i$ with $y_1
    =\ell_1$ and $y_2 = \ell_2$. We call a quadric $q$ \emph{special} if $q(0,
    0, y_3, y_4)$ has rank one. Special quadrics form a divisor, hence the
    graded inverse systems with $q$ special are at most $13-1 = 12$
    dimensional and the corresponding nongraded systems are at most
    $12 + 18 = 30$ dimensional by the tangent estimate above. Thus below we
    may and do consider only non-special quadrics.
    For such a quadric $q$ take a $\Gmult$-action with weights of $y_1$, $y_2$
    equal, weights of $y_3$ and $y_4$ also equal and smaller that the
    weights of $y_1$. The obtained $\Gmult$-limit of the inverse system $y_1^4$,
    $y_2^3$, $q$ is $y_1^4$, $y_2^3$, $y_3y_4$.
    The positive part of the tangent space at $y_1^4$, $y_2^3$, $y_3y_4$ is
    $17$. Hence using the base and tangent-to-fiber estimate we obtain $17 + 13 = 30$.
\end{proof}
\begin{remark}
    In the above proof, one can check that the dimension of the stratum of
    $y_1^4$, $y_2^3$, $y_3^2$, $y_4$ is exactly $30$, so the punctual Hilbert
    scheme is reducible.
\end{remark}

    \begin{proof}[Proof of Theorem~\ref{ref:dimension:maintheorem}]

        Recall that the expected dimension of $\Hilb_k(\mathbb{A}^n, 0)$
        is $(k-1)(n-1)$ and a locus $\mathcal{Z}$ in this Hilbert scheme is negligible
        if $\dim \mathcal{Z}\leq (k-1)(n-1)$.
        Suppose that $H$ is such that the stratum of $H$ is not negligible.
        By Invariance of Codimension we reduce to $H(1) = n$~\cite[Proposition~A.4]{bjjm}.
        By irreducibility of $\Hilb_k(\mathbb{A}^2, 0)$, we have $H(1)\geq 3$. By Theorem~\ref{ref:H2:thm}
        we have $H(2)\geq 3$.

        Suppose now $k\leq 8$. By
        Theorem~\ref{ref:irreducibility}, we have $H(1)\geq 4$.
        As $k\leq 8$ this implies that $H = (1, 4, 3)$ and we conclude it is
        negligible by direct computation of the dimension of this locus, which
        is $\dim \Gr(3, 10) =
        21$. So, for $k\leq 8$ and
        any $n$ the dimension of the Hilbert scheme is the expected one.
        The stratum of $(1, n, 3)$ for $n\geq 5$ has
        dimension higher than excepted, see Example~\ref{ex:taugeq3} below. This concludes the part $\tau\geq 3$.

        We assume $\tau\leq 2$ and $k$ arbitrary.
        We first discard several classes of Hilbert functions. If $H(3) = 0$
        then $H(2) \leq 2$ since $\tau \leq 2$, a contradiction with the above.
        If $H(3) = 1$ then Corollary~\ref{ref:H3eq1negligible:cor} shows that
        $H$ is negligible.
        If $H = (1, n, 3, 2)$ for some $n$, then
        Proposition~\ref{ref:1m32new:prop}
        concludes.

        Suppose now $k \leq 11$. Using Theorem~\ref{ref:irreducibility} we get
        $H(1)\geq 4$. In conjunction with the above, we have $H(1)\geq 4$, $H(2)\geq 3$ and
        $H(3)\geq 2$. This immediately concludes the case $k = 9$. For $k =
        10$ this leaves, a priori, the single possible case $H = (1,4,3,2)$
        but also that case was considered above.

        Suppose $k=11$. Due to the constraints above, we have to consider only
        the cases $H = (1,5,3,2)$, $H = (1, 4, 4, 2)$, $H = (1,4,3,3)$, $H =
        (1,4,3,2,1)$. The case $(1,5,3,2)$ was considered above.
        The case $(1,4,4,2)$ was considered in
        Proposition~\ref{ref:1442:prop}. The case $(1,4,3,3)$ is
        impossible to obtain with $\tau\leq 2$.
        The case $(1,4,3,2,1)$ is considered in Proposition~\ref{ref:14321final:prop}.
        This concludes the case $k=11$ and the whole proof of the bound. The
        violations of the bound for higher $k$ follow from
        Examples~\ref{ex:tauone}-\ref{ex:tautwo}.
    \end{proof}
    \begin{remark}
        It is natural to speculate about the cases $\tau = 1$, $k=12,13$ and
        $\tau = 2$, $k=12$. We believe that in these cases the scheme
        $\Hilb_k^{\tau}(\mathbb{A}^n, 0)$ has the expected dimension, however the
        number of cases to consider would increase the length of the paper too
        greatly and ruin the relative cleanness of the current proof.
    \end{remark}
    \begin{proof}[Proof of Theorem~\ref{ref:mainthm}]
        The punctual Hilbert schemes $\Hilb_k(\mathbb{P}^n, p)$ and
        $\Hilb_k(\mathbb{A}^n, 0)$ have the same dimension~\cite[Proposition~2.2]{fogarty}. Since
        $\dim\HilbSmk(\mathbb{A}^n) = kn$, and the map
        $\mathbb{A}^n\times\HilbSmk(\mathbb{A}^n,
        0)\to  \HilbSmk(\mathbb{A}^n) $ is
        injective and not dominant,
        we have $\dim \HilbSmk(\mathbb{A}^n, 0)
        \leq (k-1)n-1$. Moreover, if $k\leq 8$, or if $\tau \leq 2$ and $k\leq 11$,
        then it follows from Theorem~\ref{ref:dimension:maintheorem} that
       $\dim \Hilb_k^{\tau}(\mathbb{A}^n, 0) = (k-1)(n-1)$.
        Therefore the
        estimates for the theorem follow from
        Proposition~\ref{ref:reductionOfN:prop} in the general case and from
        Corollary~\ref{ref:bound:cor} in the special cases $k\leq 8$ or
        $\tau\leq 2$, $k\leq
        11$.
    \end{proof}

    \subsection{Counterexamples}\label{sec:counterexamples}

    \begin{example}[$\tau \geq 3$]\label{ex:taugeq3}
        The locus of algebras with Hilbert function $H = (1,n,3)$ in $\mathbb{A}^n$ is
        parametrized by $\Gr\left(3, \binom{n+1}{2}\right)$ whence it has
        dimension $3\cdot \left( \binom{n+1}{2} - 3 \right)$, whereas the
        expected dimension is $(n-1)\cdot (n+3)$. The difference is
        \[
            \binom{n-1}{2} - 6
        \]
        which is positive for $n \geq 5$; so for $k = n+4 \geq 9$. To
        get examples for $n > k-4$ use Invariance of
        Codimension~\cite[Proposition~A.4]{bjjm}.
    \end{example}

    \begin{example}[$\tau = 1$]\label{ex:tauone}
        Consider the Hilbert functions $(1,n,n,1)$ and $(1,n,n,1,1)$.
        By Proposition~\ref{ref:locusOfGorensteins:prop} the dimensions of the
        corresponding loci minus the expected dimension $(n-1)(k-1)$ is in
        both cases
        \[
            \binom{n+2}{3} - 1 + \binom{n+2}{2} - n\cdot (2n-1) =
            \frac{1}{6}\,n^{3}-n^{2}+\frac{5}{6}\,n = \frac{1}{6}n(n-1)(n-5)
        \]
        which is strictly positive for every $n\geq 6$ hence for $k\geq 14$.
    \end{example}

    \begin{example}[$\tau = 2$]\label{ex:tautwo}
        Take an inverse system for two general polynomials in $5$ variables: one
        of degree $3$ and one of degree $2$. Then the Hilbert function is
        $(1,5,6,1)$ and the optimal bound is $4\cdot 12 = 48$. But
        the space of homogeneous such choices is
        $\binom{7}{3}-1+(10-1) = 43$ dimensional and so the whole space is
        $52$-dimensional, violating the bound.
        (This example could be generalized to $(1,n,n+1,1)$.)
    \end{example}

    {\small
\newcommand{\etalchar}[1]{$^{#1}$}

}

\end{document}